\documentclass[preprint,12pt]{elsarticle}
\usepackage{amsthm}
\usepackage{amsmath,amscd}
\usepackage{cleveref}
\usepackage{mathtools}

\DeclarePairedDelimiter\floor{\lfloor}{\rfloor}
\usepackage{amsfonts}
\usepackage{graphicx}
\usepackage{bbm}
\usepackage[dvipsnames]{xcolor}
\usepackage{enumerate}
\usepackage{comment}

\numberwithin{equation}{section} 
\newcommand{\bean}{\begin{eqnarray*}}
\newcommand{\eean}{\end{eqnarray*}}
\newcommand{\be}{\begin{equation}}

\newcommand{\ee}{\end{equation}}
\newcommand{\bd}{\begin{displaymath}}
\newcommand{\ed}{\end{displaymath}}
\newcommand{\eps}{\varepsilon}

\newcommand{\beq}{\begin{equation}}
\newcommand{\eeq}{\end{equation}}
\newcommand{\bea}{\begin{eqnarray}}
\newcommand{\eea}{\end{eqnarray}}

\newcommand{\curl}{\mathrm{curl}\,}
\newcommand{\dive}{\mathrm{div}\,}

\newtheorem{lemma}{Lemma}[section]
\newtheorem{theorem}[lemma]{Theorem}
\newtheorem{corollary}[lemma]{Corollary}
\newcommand{\abs}[1]{\left\vert{#1}\right\vert}

\newcommand{\R}{\mathbb{R}}
\newtheorem{remark}[lemma]{Remark}
\newcommand{\e}{\varepsilon}
\newcommand{\norm}[1]{\left\Vert#1\right\Vert}
\newcommand{\C}{\mathbb{C}}
\newtheorem{Theorem}[lemma]{Theorem}
\begin{document}
\begin{frontmatter}
\title{A One-Dimensional Variational Problem for Cholesteric Liquid Crystals with Disparate Elastic Constants}
\author[label1]{Dmitry Golovaty}
\author[label2]{Michael Novack}
\author[label3]{Peter Sternberg}
\address[label1]{Department of Mathematics, The University of Akron, Akron, OH 44325}
\address[label2]{Department of Mathematics,
The University of Texas at Austin
2515 Speedway, RLM 8.100,
Austin, TX 78712}
\address[label3]{Department of Mathematics, Indiana University, Rawles Hall
831 East 3rd St.,
Bloomington, IN 47405}
\begin{keyword}
Cholesteric liquid crystal \sep Gamma-convergence \sep local minimizer
\end{keyword}
\begin{abstract}
    We consider a one-dimensional variational problem arising in connection with a model for cholesteric liquid crystals. The principal feature of our study is the assumption that the twist deformation of the nematic director incurs much higher energy penalty than other modes of deformation. The appropriate ratio of the elastic constants then gives a small parameter $\varepsilon$ entering an Allen-Cahn-type energy functional augmented by a twist term. We consider the behavior of the energy as $\varepsilon$ tends to zero. We demonstrate existence of the local energy minimizers classified by their overall twist, find the $\Gamma$-limit of the relaxed energies and show that it consists of the twist and jump terms. Further, we extend our results to include the situation when the cholesteric pitch vanishes along with $\varepsilon$.  
\end{abstract}
\end{frontmatter}
\section{Introduction}
We seek an understanding of the energy landscape for the one-dimensional variational problem
\begin{equation}
    \label{eq:vp}
    \inf_{\mathcal{A}_\alpha}E_\e(u),
\end{equation}
where $u:[0,1]\to\R^2$ so that $u=(u_1,u_2)$ with
\begin{eqnarray}
&&
E_\e(u_1,u_2)=\int_0^1 \frac{\e}{2}\abs{u'}^2+\frac{1}{4\e}(\abs{u}^2-1)^2+
\frac{L}{2}(u_1\,u_2'-u_2\,u_1'-2\pi N)^2\,dx,\label{Eeps}\\
&&\mbox{and} \nonumber\\
&&
\mathcal{A}_\alpha:=\{u\in H^1((0,1);\R^2):\;
u(0)=1,\; u(1)=e^{i\alpha}\},\label{bc}
\end{eqnarray}
for some positive integer $N$ and some $\alpha\in [0,2\pi)$

When convenient, as above, we will view $u=(u_1,u_2)$ as a map into $\C$. On occasion we will also find it convenient to use the following notation for the twist term:
\[
\mathcal{T}(u):=u_1\,u_2'-u_2\,u_1'.
\]

Our purpose in this article is to continue the analysis of a family of models with disparate elastic constants arising in the mathematics of liquid crystals \cite{GNS,GNS2,GNSV,GSV}. In particular, the problem \eqref{eq:vp} can be viewed as a highly simplified, relaxed version of the Oseen-Frank model for {\it cholesteric liquid crystals}, {\cite{bernardino2014structure,PhysRevLett.46.1216,ravnik,selinger,smalyukh,taylor2020gamma}} based on the elastic deformations of an $\mathbb{S}^1$- or $\mathbb{S}^2$-valued director $n$, cf. \cite{Virgabook}. Other models, of course, exist for nematic liquid crystals, including  the $Q$-tensor based Landau-de Gennes model, whose energy density consists of a bulk potential favoring either a uniaxial nematic state, an isotropic state, or both, depending on temperature, cf. \cite{MN}. We refer the reader to the recent literature \cite{GNS,MajZar} that establishes a precise asymptotic relationship between the Oseen-Frank and the Landau-de Gennes models.

 We recall now the form of the Oseen-Frank energy,
 \begin{align}\notag
F_{OF}(n):=&\int_\Omega\left(\frac{K_1}{2}(\dive n)^2 + \frac{K_2}{2}((\curl n) \cdot n+q)^2 +\frac{K_3}{2}|(\curl n )\times n|^2\right. \\ \label{FOF1}
&\left. \quad\quad +\frac{K_2+K_4}{2}(\textup{tr}\, (\nabla n)^2 - (\dive n)^2)\right)\,dx,
\end{align}
where $\Omega\subset\R^3$ represents the sample domain and the director $n$ maps 
$\Omega$ to $\mathbb{S}^2$. The material constants $K_1,K_2,K_3$ and $K_4$ are the elastic coefficients associated with the deformations of splay, twist, bend and saddle-splay, respectively \cite{Virgabook}. Most important for this article is the second term, the twist, where $q=\frac{2\pi}{p}$ with $p$ being the pitch of the cholesteric helix.  The distinction between nematic and cholesteric liquid crystals is manifested by the value of $q$. The liquid crystal is in a nematic state when $q=0$ and, absent boundary conditions, a global minimizer of $F_{OF}$ is a constant director field. On the other hand, a liquid crystal is in a cholesteric state whenever $q\neq0$ and global minimizers of $F_{OF}$ in $\mathbb R^3$ are rigid rotations of a uniformly twisted director field $n=(n_x,n_y,0)=e^{\frac{2\pi iz}{p}}$.

In \cite{GSV} we propose and analyze a model problem for nematic liquid crystals carrying a large energetic cost for splay. The model couples the Ginzburg-Landau potential to an elastic energy density with large elastic disparity, namely
\begin{equation}
\inf_{u\in H^1(\Omega;\R^2)}\frac{1}{2} \int_\Omega \left(\e|\nabla u|^2  +L(\dive u)^2 +\frac{1}{\e}(1-|u|^2)^2\right) \,dx.\label{divBBH}
\end{equation}
Here one should view $L$ as playing a role analogous to $K_1$ in \eqref{FOF1}. The minimization is taken over competitors satisfying an $\mathbb{S}^1$-valued Dirichlet condition on $\partial\Omega$ so as to avoid a trivial minimizer. This choice of potential clearly favors $\mathbb{S}^1$-valued states, 
which are a stand-in in our models for uniaxial nematic states. Analysis of \eqref{divBBH} in the $\e\to 0$ limit involves a `wall energy' along a jump set $J_u$ penalizing jumps of any $\mathbb{S}^1$-valued competitor $u$, and bulk elastic energy favoring low divergence. The conjectured $\Gamma$-limit of \eqref{divBBH} is
\beq
\frac{L}{2} \int_{\Omega} (\dive u)^2 \,dx + \frac{1}{6}\int_{J_{u}\cap \Omega} |u_+ - u_-|^3 \,d \mathcal{H}^1,\label{limdivBBH}
\eeq
where $u_+$ and $u_-$ are the one-sided traces of $u$ along $J_u$ which exhibit a jump discontinuity in their tangential components. 

The model considered in this paper is a cholesteric analog of the problem in \cite{GSV}. Just as the functional considered in \cite{GSV} can be viewed as a Ginzburg-Landau-type relaxation of the splay $K_1-$term in \eqref{FOF1}, the problem \eqref{eq:vp} can be understood as a similar relaxation of the twist $K_2-$term in the same energy. For example, in 2D this relaxation may take the form
\begin{equation}
    \label{eq:vp2}
    \inf_{\mathcal{A}}E^{2D}_\e(u),
\end{equation}
where $u:\Omega\to\R^3$ with
\begin{eqnarray}
&&
E^{2D}_\e(u)=\int_\Omega \frac{\e}{2}\abs{\nabla u}^2+\frac{1}{4\e}(\abs{u}^2-1)^2+
\frac{L}{2}(u\cdot\curl{u}-2\pi N)^2\,dx,\label{Eeps2}\\
&&\mbox{and} \nonumber\\
&&
\mathcal{A}:=\{u\in H^1(\Omega;\R^3):\;
u|_{\partial\Omega}=u_0\},\label{bc2}
\end{eqnarray}
for some domain $\Omega\subset\R^2$, some positive integer $N$ and boundary condition $u_0:\partial\Omega\to \mathbb{S}^2$. Results of simulations for the gradient flow dynamics associated with the problem \eqref{eq:vp2} lead to intricate textures, {such as that} shown in Fig.~\ref{fig}, resembling cholesteric fingerprint textures observed in experiments \cite{Outram}.

While attempting to tackle the problem \eqref{eq:vp2}, we found that the energy landscape in \eqref{eq:vp} is already rich enough to merit a separate investigation in one dimension that we undertake in this paper. We further assume that the component of $u$ along the axis of the twist vanishes so that the target space for the director is two-dimensional. Thus, though we will write $u=(u_1(x),u_2(x))$ what we really have in mind is $u=(0,u_2(x),u_3(x)).$  The thought experiment that allows us to impose this condition assumes that an electric field is applied along the axis of the twist and that the cholesteric has negative dielectric anisotropy that forces its molecules to orient perpendicular to the field, \cite{anisot}. 

{Existence and stability of minimizers for the three-component cholesteric director within the framework of the Oseen-Frank model in one dimension was considered in \cite{bedford2014global} and  \cite{Gartland2010} under the assumption that all elastic constants have comparable values. In addition, in \cite{Gartland2010}, the energy functional included the effects of an electric field. In the one-dimensional setting for highly disparate elastic constants,}  it turns out the inclusion of a third $x$-dependent component leads to an energy where distinguishing textures are lost for $\e\ll 1$ and the energy landscape becomes {highly degenerate}, see Remark \ref{thirdcomp}.  Thus, we find that the one-dimensional, two-component model \eqref{Eeps}  leads to stable states more reminiscent of those described above for the two-dimensional problem. 

\begin{figure}
\begin{center}
\includegraphics[scale=1]{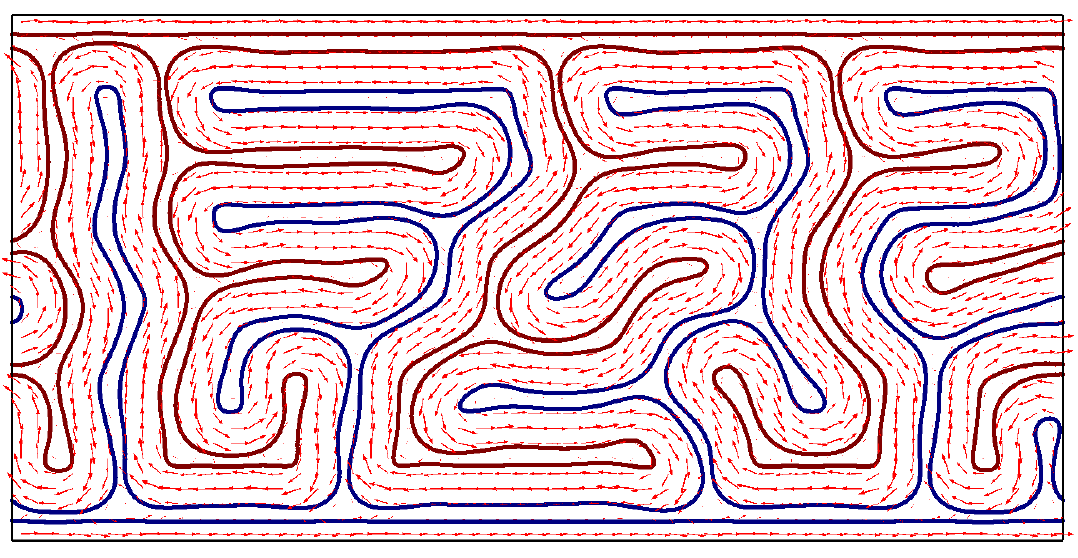}
\caption{Numerical solution for the gradient flow associated with \eqref{eq:vp2} obtained in COMSOL \cite{COMSOL}. The arrows represent the director $u$, the blue and the red curves are level sets $u_3=-0.92$ and $u_3=0.92$, respectively. The simulation was started from a uniform twist state with the axis of the twist oriented in a vertical direction. The director is assumed to be oriented to the right and to the left on the top and the bottom boundaries, respectively. Periodic boundary conditions are imposed on vertical components of the boundary. Here $N=10$, $L=1$, and $\e=0.005$.}\label{fig}
\end{center}
\end{figure}

The richness of the energy landscape is first revealed in Section 2 where the key result is Theorem \ref{constrainedmins}, showing that local minimizers of $E_\e$ exist for every positive integer value of twist--essentially for every winding number.

Section 3 contains our principal result of this investigation, namely that similar to our work on \eqref{divBBH} in \cite{GSV}, the $\Gamma-$limit $E_0$ given by \eqref{Ezero} of the relaxed energy $E_\e$ is the twist energy defined over $\mathbb S^1-$valued maps along with a jump energy, cf. Theorems \ref{gconv} and \ref{cpt}. One distinction, however, between our $\Gamma$-limit here and \eqref{limdivBBH} is that in the present study the jump cost, now associated with jumps in the phase, is impervious to the {\it size} of the jump. We demonstrate in Theorem \ref{e0mins} and Corollary \ref{threepos} that in certain parameter regimes depending on $L$ and $\alpha$, global energy minimizers with jumps are energetically favorable. Indeed, this is the most dramatic effect of the assumption of disparate elastic constants present in our model. The relatively expensive cost of twist leads the global minimizer of \eqref{eq:vp}, which of course is necessarily smooth, to rapidly change its phase, a process that can only be achieved with finite energetic cost by having the modulus simultaneously plunge towards zero.

In Section 4 we establish an energy barrier between the local minimizers of different winding numbers exposed in Theorem \ref{constrainedmins}, cf. Theorem \ref{thm:sublevel}. This readily leads to the existence of saddle points in Theorem \ref{MPT} via the Mountain Pass Theorem, thus filling out the energy landscape for $E_\e$.

Finally, in Section 5 we investigate the energy \eqref{tEeps} motivated by studies of so-called twist bend nematics, where twisting of the director occurs at much shorter scales than in cholesterics \cite{TwistBend}. Here we model this situation by tying the pitch (or the period of the twist) $1/N$ to the Ginzburg-Landau parameter $\e$ so that twisting ``averages out" in the limit $\e\to0$. We show in Theorem \ref{macro} that, in fact, the weak limit of uniformly energy bounded director fields is equal to zero but we are nonetheless able to recover some information about fine scale behavior of these fields. Then in Theorems \ref{gammablowup}
and \ref{cpt2} we establish $\Gamma$-convergence in this setting.

\section{Global and local minimizers that stay bounded away from zero}

We begin with the observation for problem \eqref{Eeps}-\eqref{bc} that a global minimizer exists for fixed $\e>0$.
\begin{Theorem}\label{GE}
For each fixed $\e>0$ there exists a minimizer of $E_\e$ within the class $\mathcal{A}_\alpha.$
\end{Theorem}
\begin{proof}
Existence follows readily from the direct method as follows.
Suppressing the $\e$-dependence, let $\{u^j\}=\{(u_1^j,u_2^j)\}$ denote a minimizing sequence:
\[
E_\e(u_1^j,u_2^j)\to m:=\{\inf E_\e(u):\;u\in \mathcal{A}_\alpha\}.
\]
Compactness of a minimizing sequence follows from the immediate energy bounds
\[
\int_0^1\abs{u^j\,'}^2\,dx<C,\quad \int_0^1\abs{u^j}^4\,dx<C,\quad
\int_0^1\big(u_1^ju_2^j\,'-u_2^ju_1^j\,'\big)^2\,dx<C.
\]
So, in particular we have a uniform $H^1$-bound on $\{u^j\}$. Thus, up to subsequences, we get 
uniform (in fact Holder) convergence of $u^j\to \bar{u}=(\bar{u}_1,\bar{u}_2),$ and
$u^j\,'\rightharpoonup \bar{u}'$ weakly in $L^2((0,1))$ for some $\bar{u}\in \mathcal{A}_\alpha.$

Turning to the issue of lower-semicontinuity, we note that  verification for the first two terms in $E_\e$ is standard. For the third term we observe that 
\[
u_1^ju_2^j\,'-u_2^ju_1^j\,'\rightharpoonup \bar{u}_1\bar{u}_2\,'-\bar{u}_2\bar{u}_1\,'\;\mbox{weakly in}\;L^2,
\]
through the pairing of weak $L^2$ and uniform convergence.

Then we have
\begin{eqnarray*}&&\int_0^1 (u_1^j\,u_2^j\,'-u_2^j\,u_1^j\,'-2\pi N)^2\,dx=\\
&&\int_0^1 (u_1^j\,u_2^j\,'-u_2^j\,u_1^j\,')^2\,dx-4\pi N\int_0^1 \big(u_1^j\,u_2^j\,'-u_2^j\,u_1^j\,'\big)\,dx
+4\pi^2N^2.
\end{eqnarray*}

The middle term is continuous given the strong convergence of $u^j$ to $\bar{u}.$ For the first term, we appeal to the lower-semicontinuity of the $L^2$ norm under weak $L^2$ convergence. Thus, $E_\e(\bar{u})=m.$
\end{proof}

It turns out that characterization of the global minimizer in the case where $\alpha=0$, so that the boundary conditions are simply $u(0)=u(1)=1$, is much simpler than when $\alpha\in (0,2\pi)$. In particular, we have the following result.
\begin{Theorem}\label{bigzero}
Let $u_\e$ denote a global minimizer of $E_\e$ within the admissible class $\mathcal{A}_0$. Then $\rho_\e(x):=\abs{u_\e(x)}$ converges to $1$ uniformly on $[0,1]$ as $\e\to 0.$
\end{Theorem} 
\begin{proof}
We proceed by contradiction and assume that for some $\delta>0$ there exists a sequence
  $\e_j\to 0$ and values $x_j\in [0,1]$ such that 
  \[\rho_{\e_j}(x_j)\leq 1-\delta.
  \]
   The case where  $\rho_{\e_j}(x_j)\geq 1+\delta$ is handled similarly.

 We begin with the observation that
\begin{equation}
E_\e(u_\e)\leq E_\e(e^{i2\pi Nx})=2(\pi N)^2\e.\label{easybd}
\end{equation}
It then follows that for some $C_0>0$ independent of $\e$ one has
\[
\int_0^1(\rho_\e')^2\,+\rho_\e^4\,dx<C_0,
\]
which in turn implies a bound of the form
\[
\norm{\rho_\e}_{H^1(0,1)}<C_1=C_1(C_0).\quad\mbox{Hence,}\quad \norm{\rho_\e}_{C^{0,1/2}(0,1)}<C_1.
\]

Then invoking the H\"older bound above, we have
\[
\abs{\rho_\e(x)-\rho_\e(x_j)}\leq C_1\abs{x-x_j}^{1/2}
\]
and so for $\abs{x-x_j}\leq \big(\frac{\delta}{2C_1}\big)^2$ one would have 
\[
\rho_\e(x)\leq \rho_\e(x_j)+C_1\abs{x-x_j}^{1/2}\leq 1-\frac{\delta}{2}.
\]
This in turn would imply
\begin{eqnarray*}
E_\e(u_\e)\geq\frac{1}{4\e}\int_0^1(\rho_\e^2-1)^2\,dx&&\geq \frac{1}{4\e}\int_{\left\{x:\,\abs{x-x_j}\leq \big(\frac{\delta}{2C_1}\big)^2\right\}}
(\rho_\e^2-1)^2\,dx\\
&&\geq \frac{\delta^4}{64 C_1^2\e}.
\end{eqnarray*}

This cannot hold in light of \eqref{easybd} for $\e< \e_0$ where 
\[
\e_0=\frac{\delta^2}{8\sqrt{2}C_1\pi N}.
\]
\end{proof}
\vskip.2in
\noindent
Next we turn to the construction of local minimizers of $E_\e$ within the class $\mathcal{A}_\alpha$ for $\alpha\in [0,2\pi)$. Like the global minimizers constructed for the case $\alpha=0$ in Theorem \ref{bigzero}, the modulus of these local minimizers will converge uniformly to $1$ as $\e\to 0$. 
\begin{Theorem}\label{constrainedmins}
For every positive integer $M$ and every $\alpha\in [0,2\pi)$, there exists an $\e_0>0$ such that for all $\e<\e_0$ there is an $H^1$-local minimizer $u_{\e,M}=\rho_{\e,M} e^{i\theta_{\e,M}}$ of $E_\e$ within the class $\mathcal{A}_\alpha$ such that
\begin{eqnarray}
&&
\limsup_{\e\to 0}\frac{\norm{{\rho_{\e,M}}-1}_{L^{\infty}(0,1)}}{\e}<\infty,
\label{suprho}\\
&&\nonumber\\
&&\lim_{\e\to 0}\theta_{\e,M}'= 2\pi M+\alpha\;\mbox{uniformly in}\;x\in [0,1],\label{thetaprime}\\
&&\mbox{and}\nonumber\\
 &&   
    \lim_{\e\to 0}E_\e(u_{\e,M})=\frac{L}{2}\left(2\pi(M-N)+\alpha\right)^2.\label{eq:ls}
    \end{eqnarray}
\end{Theorem}
\begin{remark}
We will find later that in some parameter regimes, corresponding to $\alpha$ small and $M=N$, these local minimizers turn out in fact to be global minimizers. However, when $M\not= N$ or when $M=N$ but $\alpha$ exceeds a critical value, they will not.
\end{remark}
\begin{proof}
   To capture these local minimizers we will rephrase our problem by switching to polar coordinates via the substitution
\[
u_1=\rho\cos\theta,\quad u_2=\rho\sin\theta.
\]
The boundary conditions corresponding to \eqref{bc} are
\begin{equation}
\rho(0)=1=\rho(1),\quad \theta(0)=0,\;\theta(1)=2\pi M+\alpha\quad\mbox{for some integer}\;M>0.\label{rtbc}
\end{equation}
We find that in these variables,
\[
E_\e=E_\e(\rho,\theta)=
\int_0^1 \frac{\e}{2}\big((\rho')^2+\rho^2(\theta')^2\big)+\frac{1}{4\e}(\rho^2-1)^2+
\frac{L}{2}(\rho^2\theta'-2\pi N)^2\,dx.
\]
We will minimize $E_\e(\rho,\theta)$ subject to \eqref{rtbc} via a constrained minimization procedure. To this end, 
for any number $\rho_0\in (0,1)$ we introduce the admissible class
 \begin{equation}
\mathcal{H}_{\rho_0}:=\{\rho\in H^1(0,1):\;\rho(0)=1=\rho(1),\;\rho(x)\geq \rho_0\;\mbox{on}\;[0,1]\}\label{mtH}
\end{equation}
and for any positive integer $M$ and any $\alpha\in [0,2\pi)$ we denote
\begin{equation}
\mathcal{H}_{M,\alpha}:=\{\theta\in H^1(0,1):\;\theta(0)=0,\;\theta(1)=2\pi M+\alpha\}.
\label{TM}
\end{equation}

We note that for each fixed $\e>0$ and $\rho_0\in (0,1)$, the direct method provides for a minimizing pair $({\rho_{\e,M}},{\theta_{\e,M}})$ to the constrained problem:
\begin{equation}
{\mu_{\e,M}}:=\inf_{\rho\in\mathcal{H}_{\rho_0},\,\theta\in\mathcal{H}_{M,\alpha}}E_\e(\rho,\theta).\label{Mmin}
\end{equation}
The only point to be made here is that the lower bound $\rho_j\geq \rho_0$ on a minimizing sequence $\{\rho_j,\theta_j\}$ allows for $H^1$ control of $\{\theta_j\}$. Also the $H^1$ control on $\{\rho_j\}$ yields uniform convergence of a subsequence so that the constraint is satisfied by the limiting ${\rho_{\e,M}}$. 

We remark for later use that ${\mu_{\e,M}}$ is bounded independent of $\e$ since 
\begin{equation}
{\mu_{\e,M}}\leq E_\e(1,(2\pi M+\alpha)x)=\frac{L}{2}\left(2\pi(M-N)+\alpha\right)^2+O(\e)\label{test}
\end{equation}

We will now argue that for any integer $M>0$ and any $\rho_0\in (0,1)$, these solutions to the constrained problem in fact satisfy $\rho_{\e,M}(x)>\rho_0$ for all $x\in [0,1]$ when $\e$ is sufficiently small. Hence, they correspond to $H^1$-local minimizers of $E_\e(u)$ subject to the boundary conditions \eqref{bc} since the representation $u_{\e,M}=\rho_{\e,M}e^{i\theta_{\e,M}}$ is global.

{\bf CLAIM:} For any positive integer $M$, any $\alpha\in [0,2\pi)$, and any $\rho_0\in  (0,1)$ we have
\begin{equation}
 \rho_{\e,M}(x)>\rho_0\;\mbox{for all}\;x\in [0,1]\quad\mbox{provided}\;\e\;\mbox{is sufficiently small.}\label{claim}
\end{equation}
To pursue this claim, we first observe that since the constraint falls only on ${\rho_{\e,M}}$, this minimizing pair $({\rho_{\e,M}},{\theta_{\e,M}})$ must satisfy
\begin{equation}
\lim_{t\to 0^+}\frac{E_\e\big(\rho_{\e,M}+tf,\theta_{\e,M}\big)-
E_\e\big(\rho_{\e,M},\theta_{\e,M}\big)}{t}
\geq 0,\label{rineq}
\end{equation}
for all $f\in H^1_0(0,1)$ such that $f(x)\geq 0$ on $[0,1]$,
and
\begin{equation}
\frac{d}{dt}_{t=0}E_\e\left({\rho_{\e,M}},{\theta_{\e,M}}+t\psi\right)= 0\quad\mbox{for all}\;\psi\in H^1_0(0,1).\label{tEL}
\end{equation}
Computing these quantities we find that \eqref{rineq} takes the form
\begin{multline}
\int_0^1\e{\rho^\prime_{\e,M}}f'+ \left(\e {\rho_{\e,M}}{\left({\theta^\prime_{\e,M}}\right)}^2+\frac{1}{\e}\left(\rho^2_{\e,M}-1\right){\rho_{\e,M}}\right.\\\left.-2L\left(2\pi N-{\rho^2_{\e,M}}{\theta^\prime_{\e,M}}\right){\rho_{\e,M}}{\theta^\prime_{\e,M}}\right)f\,dx\geq 0\label{rhoEL}
\end{multline}
for all nonnegative $f\in H^1_0(0,1)$, and \eqref{tEL} takes the form
\begin{equation}
\left[\left(\e\theta^\prime_{\e,M}-L\left(2\pi N-{\rho^2_{\e,M}}{\theta^\prime_{\e,M}}\right)  \right){\rho^2_{\e,M}}\right]'=0.
\end{equation}
Thus,
\begin{equation}
\big(\e{\theta^\prime_{\e,M}}-L(2\pi N-{\rho^2_{\e,M}}{\theta^\prime_{\e,M}}  ) \big){\rho^2_{\e,M}}=C_\e\quad\mbox{for some constant}\;C_\e,
\end{equation}
allowing us to solve for ${\theta^\prime_{\e,M}}$ to find
\begin{equation}
{\theta^\prime_{\e,M}}=\frac{2\pi N L{\rho^2_{\e,M}}+C_\e}{L{\rho^4_{\e,M}}+\e{\rho^2_{\e,M}}}.
\label{tMform}
\end{equation}
Integrating \eqref{tMform} over the interval $[0,1]$ and using the boundary conditions on ${\theta_{\e,M}}$ we obtain a formula for $C_\e$:
\begin{equation}
C_\e=\frac{ 2\pi M+\alpha-2\pi LN\int_0^1 (L{\rho^2_{\e,M}}+\e)^{-1}\,dx}{\int_0^1(L{\rho^4_{\e,M}}+\e{\rho^2_{\e,M}})^{-1}\,dx}.
\label{Ceqn}
\end{equation}
Now by \eqref{test},
\[
\int_0^1({\rho^2_{\e,M}}-1)\abs{{\rho^\prime_{\e,M}}}\,dx\leq\sqrt{2}\int_0^1 \frac{\e}{2}({\rho^\prime_{\e,M}})^2+\frac{1}{4\e}({\rho^2_{\e,M}}-1)^2 
\,dx\leq \sqrt{2}\,{\mu_{\e,M}}.
\]
Since ${\rho_{\e,M}}(0)=1$, it then follows from \eqref{test} and this total variation bound that ${\rho_{\e,M}}$ is bounded above uniformly in $\e$. Thus, by \eqref{Ceqn}, the same is true of $\abs{C_\e}$.

Next we use \eqref{tMform} to find that
\begin{eqnarray*}
&&{\theta^\prime_{\e,M}}-\bigg(\frac{2\pi NL+C_\e}{L+\e}\bigg)=
\frac{2\pi N L{\rho^2_{\e,M}}+C_\e}{L{\rho^4_{\e,M}}+\e{\rho^2_{\e,M}}}-
\bigg(\frac{2\pi NL+C_\e}{L+\e}\bigg)\\
&&
=\bigg(\frac{2\pi NL^2{\rho^2_{\e,M}}+C_\e\left[L(1+{\rho^2_{\e,M}})+\e\right]}{{\rho^2_{\e,M}}(L{\rho^2_{\e,M}}+\e)(L+\e)}\bigg)(1-{\rho^2_{\e,M}})=:\Lambda_\e(1-{\rho^2_{\e,M}})
\end{eqnarray*}
where $\abs{\Lambda_\e}\leq C=C(N,M,L)$ independent of $\e$ by the uniform bounds on $C_\e$ and ${\rho_{\e,M}}$.
Hence,
\begin{eqnarray}
&&\int_0^1\abs{{\theta^\prime_{\e,M}}-\bigg(\frac{2\pi NL+C_\e}{L+\e}\bigg)}\leq C\int_0^1 (1-{\rho^2_{\e,M}})\,dx\nonumber\\
&&\leq 2C\sqrt{\e}\left(\int_0^1\frac{1}{4\e}(1-{\rho^2_{\e,M}})^2\,dx\right)^{1/2}\leq 2C\sqrt{{\mu_{\e,M}}}\sqrt{\e}.\label{diff}
\end{eqnarray}

Since 
\[
2\pi M+\alpha=\int_0^1\left({\theta^\prime_{\e,M}}-\bigg(\frac{2\pi NL+C_\e}{L+\e}\bigg)\right)\,dx+
\frac{2\pi NL+C_\e}{L+\e}
\]
we can then invoke \eqref{diff} to conclude that
\begin{equation}
C_\e=2\pi L(M-N)+L\alpha+O(\sqrt{\e}).\label{Ceform}
\end{equation}
Substituting this back into \eqref{tMform} we find
\begin{equation}
{\theta^\prime_{\e,M}}=\frac{2\pi LM+L\alpha+2\pi LN({\rho^2_{\e,M}}-1)}{L\rho_{\e,M}^4+\e {\rho^2_{\e,M}}}+O(\sqrt{\e}).\label{beth}
\end{equation}

With these estimates we can now establish Claim \eqref{claim}.

In light of the boundary conditions, we need only consider $x\in (0,1).$ First, suppose by contradiction, that $\{x:\,{\rho_{\e,M}}=\rho_0\}$ contains an isolated point $x_0\in (0,1)$.
Since the obstacle in \eqref{Mmin} is smooth, it follows from standard regularity theory of obstacle problems 
(see e.g. \cite{PSU}) that ${\rho_{\e,M}}$ makes $C^{1,1}$ contact with the obstacle $y(x)\equiv 1$. However, we also have that
${\rho_{\e,M}}$ satisfies the Euler-Lagrange equation on either side of $x_0$, that is, 
\begin{equation}
\e\rho^{\prime\prime}_{\e,M}= \e {\rho_{\e,M}}({\theta^\prime_{\e,M}})^2+
\frac{1}{\e}({\rho^2_{\e,M}}-1){\rho_{\e,M}}-2L(2\pi N-{\rho^2_{\e,M}}{\theta^\prime_{\e,M}}){\rho_{\e,M}}{\theta^\prime_{\e,M}}\label{rhoELeqn}
\end{equation}
cf. \eqref{rhoEL}. Consequently the limits $x\to x_0^+$ and $x\to x_0^-$ agree for ${\rho^{\prime\prime}_{\e,M}}(x)$ so we find
that in fact ${\rho_{\e,M}}\in C^2$ in a neighborhood of $x_0$ with 
\[
{\rho^{\prime\prime}_{\e,M}}(x_0)=\e ({\theta^\prime_{\e,M}}(x_0))^2
+\frac{1}{\e}(\rho^2_0-1)\rho_0-2L(2\pi N-{\theta^\prime_{\e,M}}(x_0)){\theta^\prime_{\e,M}}(x_0).
\]
Invoking \eqref{beth} evaluated at $x=x_0$, we see \begin{equation}{\theta^\prime_{\e,M}}\sim \frac{2\pi M+\alpha+    2\pi N(\rho_0^2-1)}{\rho_0^4}+   O(\sqrt{\e})\label{thetaMN}
\end{equation} so that
\begin{equation}
{\rho^{\prime\prime}_{\e,M}}(x_0)\sim 
\frac{1}{\e}(\rho^2_0-1)\rho_0+O(1)\label{conrho}
\end{equation}
But since ${\rho_{\e,M}}$ has a minimum at $x_0$, this contradicts the requirement that ${\rho^{\prime\prime}_{\e,M}}(x_0)\geq 0$ when $\e$ is sufficiently small.

Next we suppose by way of contradiction that $\{x:\rho_{\e,M}=\rho_0\}$ contains an interval $I\subset [0,1]$. Fix a smooth non-negative function $f$ compactly supported in $I$. Then by \eqref{rhoEL} we must have
\[
\int_I\bigg(\e ({\theta^\prime_{\e,M}})^2
+\frac{1}{\e}(\rho^2_0-1)\rho_0-2L(2\pi N-{\theta^\prime_{\e,M}}){\theta^\prime_{\e,M}}\bigg)f\,dx\geq 0,
\]
again leading to a contradiction for $\e$ small. Claim \eqref{claim} is established and the local minimality of $u_{\e,M}$ follows.

We remark in passing that for the case $M<N$, one can establish the stronger statement that in fact $\rho_{\e,M}(x)>1$ for all $x\in (0,1)$ by choosing $\rho_0=1$ in the definition of the constrained set \eqref{mtH}. Then the same contradiction argument works with \eqref{thetaMN} replaced by
\[{\theta^\prime_{\e,M}}\sim 2\pi M+\alpha+O(\sqrt{\e})\] and \eqref{conrho}
replaced by \[
{\rho^{\prime\prime}_{\e,M}}(x_0)\sim -2 L\big(2\pi(N-M)-\alpha\big)\big(2\pi M+\alpha\big)+O(\sqrt{\e}).
\]

Finally, in light of the uniform in $\e$ bound on $\theta_{\e,M}'$ provided by \eqref{beth}, we observe that for any fixed values of $M$ and $N$,  the minimizing ${\rho_{\e,M}}$ must satisfy \eqref{suprho}, since otherwise, a presumed maximum of ${\rho_{\e,M}}$ at $x_0$ that is bigger than $1$ or a presumed minimum that is less than $1$ would violate \eqref{rhoELeqn}. Then applying \eqref{suprho} to \eqref{beth}, we obtain \eqref{thetaprime} as well.
We then may conclude that 
\begin{eqnarray*}
\liminf_{\e\to 0}
E_\e(\rho_{\e,M},\theta_{\e,M})&\geq&
\liminf_{\e\to 0}\frac{L}{2}\int_0^1  (\rho_{\e,M}^2\theta_{\e,M}'-2\pi N)^2\,dx\\
&=&\frac{L}{2}\left(2\pi(M-N)+\alpha\right)^2,
\end{eqnarray*}
and so \eqref{eq:ls} follows, in view of \eqref{test}.
\end{proof}

\section{$\Gamma$-convergence of $E_\e$}\label{sec:gamma}
As we shall see, the local minimizers described in Theorem \ref{constrainedmins} are also global minimizers only in certain parameter regimes. In order to fill out the characterization of global minimizers in all parameter regimes, we will turn to the machinery of $\Gamma$-convergence.

Our candidate for a limiting functional will be infinite unless
$u\in H^1((0,1)\setminus J;S^1)$ where $J$ is a finite collection of points, say  $0< x_1<x_2<\ldots<x_k<1$ for some non-negative integer $k$, along with perhaps $x=0$ and/or $x=1$ depending on whether or not the traces of $u$ satisfy the desired boundary conditions inherited from $E_\e$; that is, we include $x=0$ in $J$ only if $u(0^+)\not=1$ and we include $x=1$ in $J$ only if $u(1^-)\not=e^{i\alpha}$. 
For such a $u$ we will assume $J$ is the minimal such set of points, meaning that if any point in $J\cap (0,1)$ were eliminated, the function $u$ would no longer represent an $H^1$ function in the compliment of the smaller set of points. In particular, if $u\in H^1((0,1))$ and has the proper traces, then $J=\emptyset.$

Then we define $E_0:L^2\big((0,1);\R^2\big)\to\R$ via
\begin{align}
\label{Ezero}
E_0(u) := \left\{ 
\begin{array}{cc}
\displaystyle     \frac{L}{2}\int_0^1\displaystyle (u_1\,u_2'-u_2\,u_1'-2\pi N)^2\,dx + \frac{2\sqrt{2}}{3}\mathcal{H}^0(J)     &\mbox{if}\; u\in H^1((0,1)\setminus J;S^1) \\ \\
+\infty & \mbox{ otherwise}.
\end{array}
\right.
\end{align}
Here $\mathcal{H}^0$ refers to zero-dimensional Hausdorff measure, i.e. counting measure.

Then we claim:
\begin{theorem}\label{gconv}
$\{E_\e\}$ $\Gamma$-converges to $E_0$ in $L^2\big((0,1);\R^2\big)$.
\end{theorem}
We also have the following compactness result.
\begin{theorem}\label{cpt}
If $\{u_\e \}_{\e>0}$ satisfies 
\begin{equation}\label{energybound}
E_\e(u_\e)\leq C_0 <\infty,
\end{equation}
then there exists a function $u\in H^1((0,1)\setminus J';S^1)$ where $J'$ is a finite, perhaps empty, set of points in $(0,1)$ such that along a subsequence $\e_{\ell}\to 0$ one has
\begin{equation}\label{modcon}
u_{\e_\ell} \to u\;\mbox{in}\;L^2\big((0,1);\R^2\big).
\end{equation}
Furthermore, writing $u(x)=e^{i\theta(x)}$ for $\theta\in H^1((0,1)\setminus J')$, we have that for every compact set $K\subset\subset (0,1)\setminus J'$, there exists an $\e_0(K)>0$ such that for every $\e_\ell<\e_0$ one has $\abs{u_{\e_\ell}}>0$ on $K$ and there is a lifting whereby $u_{\e_\ell}(x)=\rho_{\e_\ell}(x)e^{i\theta_{\e_\ell}(x)}$ on $K$, with
\begin{equation}
\theta_{\e_{\ell}}\rightharpoonup \theta\;\mbox{weakly in}\;H^1_{loc}\big((0,1)\setminus J'\big).\label{weakcon}
\end{equation}

\end{theorem}
\begin{remark}\label{smallerJ}
It is not necessarily the case that $J'$ is minimal for $u$; that is, it can happen that $u\in H^1((0,1)\setminus J;S^1)$ for some proper subset $J\subset J'$ and in that case it is the minimal such set $J$ which one uses to evaluate the $\Gamma$-limit $E_0$ at $u$.  However, one cannot guarantee the validity of \eqref{weakcon} with $J'$ replaced by such a minimal $J$. For example, in a neighborhood of, say, $x=1/2$ whose size shrinks with $\e$, an energy-bounded sequence $\{u_\e\}$ could undergo a rapid jump in phase by $2\pi$ while the modulus of $u_\e$ plunges to zero--or even stays positive but very small-- in this neighborhood. Then the limiting $u$ could have well-behaved lifting across $x=1/2$ while for all $\e>0$, the function $u_\e$ would not. 
\end{remark}

\begin{remark}\label{thirdcomp}
The appearance of a jump set contribution to the $\Gamma$-limit $E_0$ is associated with the cost of a Modica-Mortola type transition layer for the modulus from value $1$ down to $0$ and back, accompanied by a rapid shift in the phase. If one instead considers a three-component model for $u=(u_1(x),u_2(x),u_3(x))$ then such a phase shift can be achieved with asymptotically vanishing cost by plunging $u_2(x)^2+u_3(x)^2$ to zero while compensating with $u_1(x)$ to keep $\abs{u}\approx 1$. This apparently leads to an absence of local minimizers with such meta-stable states eventually `melting' under a gradient flow to global minimizers given asymptotically by \eqref{onejump} of Theorem \ref{e0mins}  below. In fact, the degeneracy in such a three-component model is worse than just this: If one introduces cylindrical coordinates so that $(u_1,u_2,u_3)=(\rho\cos\theta,\rho\sin\theta,u_3)$ and then one writes $\rho=\cos\phi$ and $u_3=\sin\phi$ for some angle $\phi(x)$, a three-component version of $E_\e$ would take the form
\[
\frac{1}{2}\int_0^1 \e \phi'(x)^2+\e\theta'(x)^2+
L\big(\cos^2\phi(x)\theta'(x)-2\pi N\big)^2\,dx.
\]
Note then that for $\e$ small there is no control on $\phi'$, nor is there control on $\theta'$ when $\phi\approx \pi/2$.
\end{remark}

We now present the proofs of Theorem \ref{gconv} and Theorem \ref{cpt}. We will begin with the proof of Theorem \ref{cpt} since elements of it will be called upon in the proof of Theorem \ref{gconv}.
\begin{proof}[Proof of Theorem \ref{cpt}]
We fix an integer $q\geq 2$ and consider a sequence satisfying \eqref{energybound}. Denoting $\rho_\e:=\abs{u_\e}$, since $u_\e$ is $H^1$, we have that $\rho_\e$ is continuous and we may define the open sets 
\begin{equation}\notag
\mathcal{I}_\e := \{y \in [0,1]: \rho_\e(y) >1-2^{-q} \}.
\end{equation}
As open sets on the real line, each is a countable disjoint union of open intervals
\begin{equation}\notag
\mathcal{I}_\e = \cup_{m=1}^\infty \mathcal{I}_m^\e = \cup_{m=1}^\infty (a_m^\e, b_m^\e),
\end{equation}
with
\begin{equation}\notag
\rho_\e (a_m^\e) = \rho_\e (b_m^\e) =1-2^{-q}.
\end{equation}
Note that by the energy bound \eqref{energybound},
\begin{equation}\label{setconv}
\mathbbm{1}_{\mathcal{I}_\e} \to \mathbbm{1}_{(0,1)} \textup{ in }L^1((0,1)).
\end{equation}
Now we consider the open sets
\begin{equation}\notag
(0,1)\setminus \overline{\mathcal{I}}_\e = \mathring{\mathcal{I}}_\e^c.
\end{equation}
and similarly decompose $\mathring{\mathcal{I}}_\e^c$ into a countable union of intervals
\begin{equation}\notag
\cup_{m=1}^\infty (b_m^\e, a_{m+1}^\e)
.\end{equation}
Now some of the intervals $(b_m^\e, a_{m+1}^\e)$ could contain a point $c_m^\e$ such that 
\begin{equation}\notag
\rho(c_m^\e) = 2^{-q},
\end{equation}
and we collect those intervals and label them $(b_{m_j}^\e, a_{m_j+1}^\e)$, where $j$ belongs to an index set $S_\e$. A priori $S_\e$ could be finite or infinite. Let $B_\e$ be the union of these ``bad intervals." These are the intervals over which it is possible that a limit of $u_\e$ exhibits a jump discontinuity. We first prove that the number of these intervals is finite and bounded uniformly in $\e$. We observe that
\begin{align}\notag
C_0 &\geq  \int_{B_\e} \frac{\e}{2}\abs{u_\e'}^2+\frac{1}{4\e}(\abs{u_\e}^2-1)^2 \,dx\\ \notag
&\geq \sum_{j\in S_\e} \int_{b_{m_j}^\e}^{c_{m_j}^\e}\frac{\e}{2}(\rho_\e')^2 +\frac{1}{4\e}(\rho_\e^2-1)^2 \,dy +\int_{c_{m_j}^\e}^{a_{{m_j}+1}^\e}\frac{\e}{2}(\rho_\e')^2 +\frac{1}{4\e}(\rho_\e^2-1)^2 \,dy \\ \notag
&\geq \sum_{j\in S_\e}\int_{b_{m_j}^\e}^{c_{m_j}^\e}\frac{|\rho_\e'||\rho_\e^2-1|}{\sqrt{2}} \,dy +\int_{c_{m_j}^\e}^{a_{{m_j}+1}^\e}\frac{|\rho_\e'||\rho_\e^2-1|}{\sqrt{2}} \,dy \\ 
&\geq \sum_{j\in S_\e} \sqrt{2}\int_{2^{-q}}^{1-2^{-q}} |z^2-1|\,dz \label{mmest}.
\end{align}
Rearranging \eqref{mmest} yields an estimate on the size of $S_\e$:
\begin{equation}
\mathcal{H}^0(S_\e) \leq \left(\sqrt{2}\int_{2^{-q}}^{1-2^{-q}} |z^2-1|\,dz \right)^{-1}C_0.\label{notalot}
\end{equation}
Next, on $(0,1) \setminus B_\e$, we observe that
$\rho_\e \geq 2^{-q},$
which allows us define a lifting of $u_\e$ as $\rho_\e e^{i\theta_\e}$ and to find a positive constant $C_1$ such that 
\begin{align}\notag
\int_{(0,1)\setminus B_\e} (\theta_\e')^2 \, dy &\leq C_1 + C_1 \int_{(0,1)\setminus B_\e} \frac{L}{2} ( \rho_\e^2\theta_\e'-2\pi N)^2\,dy \\ 
&\leq C_1 + C_1 E_\e(u_\e)\leq C_1+C_1C_0 < \infty \label{h1}.
\end{align}
On each of the (finitely many) intervals comprising $(0,1)\setminus B_\e$ we may choose our lifting such that the value of $\theta_\e$ at, say, the left endpoint of the interval lies in $[0,2\pi)$ and from the fundamental theorem of calculus and Cauchy-Schwarz it then follows from \eqref{h1} that $\norm{\theta_\e}_{L^{\infty}((0,1)\setminus B_\e)}$ is bounded uniformly in $\e$ by a constant depending on $C_0$ and $C_1$. Consequently, we have a bound of the form 
\begin{equation}
\norm{\theta_\e}_{H^1((0,1)\setminus B_\e)}<C_2,\label{unih1}
\end{equation}
for some constant $C_2$ independent of $\e$.

Now we are going to obtain a subsequence of $\e$ approaching zero along which the bad intervals converge to a finite set of points. To this end, we start with the sequence of all the endpoints of the left-most subinterval in $B_\e$ and extract a subsequential limit, calling it $x_1$. Then, along this subsequence of $\e's$, we move on to the left endpoints of the second subinterval of $B_\e$, and passing to a further subsequence, arrive at a limit point $x_2$, etc. In light of \eqref{notalot}, this procedure generates a finite number of points $x_1<x_2\ldots<x_k$ in $[0,1]$. (If this procedure ever yields $x_j=x_{j+1}$ then we drop $x_{j+1}$ from this list.) In this manner, we arrive at a subsequence, $\e_\ell\to 0$ such that:
\begin{equation}\notag
\mathcal{H}^0(S_{\e_\ell}) \textup{ is independent of }\ell \textup{ and equal to some fixed  }k\in \mathbb{N},
\end{equation}
and, in light of \eqref{setconv}, the subintervals of $B_{\e_{\ell}}$ collapse to these $k$ points as $\e_\ell \to 0$; that is
\begin{equation}
B_{\e_\ell}\to J':= \{x_1,x_2,\ldots,x_k\}\;\mbox{as}\;\e_\ell \to 0.\label{Bcollapse}
\end{equation}. 

If we then fix any finite union of closed intervals $K_1 \subset\subset [0,1] \setminus J'$, it follows from \eqref{Bcollapse} that
\begin{equation}\label{empty}
K_1 \cap B_{\e_\ell} = \emptyset
\end{equation}
for $\e<\e_0$ with $\e_0=\e_0(K_1)$ small enough. Therefore, $u_{\e_\ell}$ has a lifting on the various intervals comprising $K_1\cap B_{\e_\ell}$ and invoking \eqref{unih1}, we have, after passing to a further subsequence, (with notation suppressed) that
\begin{equation}
\theta_{\e_\ell} \rightharpoonup \theta \quad \textup{in }H^1(K_1),\;\theta_{\e_\ell} \to \theta\quad\mbox{in}\;
L^2(K_1) \label{2con}
\end{equation}
for some $\theta\in H^1(K_1)$ such that
\begin{equation}\label{thetabd}
\|\theta \|_{H^1(K_1)} \leq C_2.
\end{equation}
Repeating this procedure on a nested sequence of sets 
\begin{equation}\label{nested}
K_1 \subset\subset K_2 \subset\subset \cdots \subset\subset K_p \subset\subset \cdots [0,1] \setminus J'
\end{equation}
which exhaust $[0,1]\setminus J'$, and passing to further subsequences via a diagonalization procedure we
arrive at a subsequence (still denoted here by $\e_\ell \to 0$) such that \eqref{weakcon} holds for some $\theta\in H^1\big((0,1)\setminus J'\big)$. 

Finally, we define $u\in H^1\big((0,1)\setminus J';S^1\big)$ via $u(x):=e^{i\theta(x)}$ and verify \eqref{modcon}. The uniform bound \eqref{energybound} implies that $\rho_\e\to 1$ in $L^2((0,1))$ and also that
\begin{equation}\label{linfbound}
C_0\geq\int_0^1\left|1-\abs{\rho_\e}^2\right|\abs{\rho_\e'}\,dx\geq \left| \int_{x_\e}^y (1-\rho_\e^2)\rho_\e'\,dx      \right|
\end{equation}
for any $y\in (0,1)$ where $x_\e\in (0,1)$ is any point selected such that, say, $\rho_\e(x_\e)\leq 2.$ It follows that $\norm{\rho_\e}_{L^\infty(0,1)}<M$ for some $M=M(C_0)$ independent of $\e$. Hence,
for any $\eta>0$ if we select a compact set $K\subset [0,1]\setminus J'$ such that $\abs{[0,1]\setminus K}<\eta$, we can appeal to \eqref{2con} to conclude \eqref{modcon} since
\begin{eqnarray*}
&&\limsup_{l\to \infty}\int_0^1\abs{u_{\e_l}-u}^2\,dx\leq \limsup_{l\to \infty}\int_K\abs{u_{\e_l}-u}^2\,dx+
\limsup_{l\to \infty}\int_{K^c}\abs{u_{\e_l}-u}^2\,dx\\
&&\leq \limsup_{l\to \infty}\int_{K^c}\abs{u_{\e_l}-u}^2\,dx\leq 2\int_{K^c}(M^2+1)\,dx<2(M^2+1)\eta.
\end{eqnarray*}
\end{proof}
\begin{proof}[Proof of Theorem \ref{gconv}]
We will first assume that
$u_\e \to u$ in $L^2\big((0,1);\R^2\big)$ and establish the inequality
\begin{equation}
\liminf E_\e(u_\e)\geq E_0(u).
\label{lscgc}
\end{equation}
To this end,  we may certainly assume that 
\begin{equation}\notag
\liminf E_\e(u_\e) \leq C_0 < \infty\quad\mbox{for some}\;C_0>0,
\end{equation}
since otherwise \eqref{lscgc} is immediate. Let $\{u_{\e_\ell} \}$ be a subsequence which achieves the limit inferior. As in \eqref{mmest} in the proof of Theorem \ref{cpt}, we can then assert that for any integer $q\geq 2$ and up to a further subsequence for which we suppress the notation, one has the lower bound
\begin{equation}
   \liminf_{\ell\to \infty} \int_{B^q_{\e_\ell}} \frac{\e_\ell}{2}\abs{u_{\e_\ell}\,'}^2+\frac{1}{4\e_\ell}(\abs{u_{\e_\ell}}^2-1)^2 \,dx
\geq \left(\sqrt{2}\int_{2^{-q}}^{1-2^{-q}} |z^2-1|\,dz\,\right) \mathcal{H}^0(J^q)  \label{newmmest}
\end{equation}
along with
\begin{equation}\theta_{\e_\ell}\rightharpoonup \theta\quad\mbox{in}\; H^1_{\rm{loc}}\big((0,1)\setminus J^q  \big).\label{thq}
\end{equation}
Here we have emphasized the $q$ dependence to write $J^q$ for the finite set of points in $[0,1]$ and $B^q_{\e_\ell}$ for the set of `bad intervals' collapsing to $J^q$ over which $\abs{u_{\e_\ell}}$ dips from values of $1-2^{-q}$ to $2^{-q}$. Next, we note that for any two positive integers $q_1<q_2$ one has the containment $B^{q_2}_{\e}\subset
B^{q_1}_\e$ and so, for any sequence $\e_{\ell}\to 0$, the finite set of points arising as the limit of $B^{q_2}_{\e_\ell}$ must be a subset of the corresponding limit of the finite collection of collapsing intervals comprising $B^{q_1}_{\e_\ell}$. Also, since the limiting phase $\theta$ of $u$ will be in $H^1_{loc}$ of the complement of any such limit of bad intervals, and since $J$ is assumed to be the minimal one, we have
\[
\mathcal{H}^0(J)\leq \mathcal{H}^0(J^q) <C_1\quad\mbox{for any}\;q<\infty,
\]
for some $C_1=C_1(C_0)$ in light of \eqref{notalot}.
Thus, passing to the limit $q\to\infty$ in \eqref{newmmest} gives
\begin{equation}\label{countlsc}
    \lim_{\ell \to \infty}\int_0^1  \frac{\e_\ell}{2}\abs{u_{\e_\ell}\,'}^2+\frac{1}{4\e_\ell}(\abs{u_{\e_\ell}}^2-1)^2 \,dx \geq \frac{2\sqrt{2}}{3}\mathcal{H}^0(J).
\end{equation}

Turning to the lower-semi-continuity of the twist term, we can repeat the argument of \Cref{cpt} to obtain that, again up to a further subsequence which we do not notate, 
\begin{equation}\theta_{\e_\ell}\rightharpoonup \theta\quad\mbox{in}\; H^1_{\rm{loc}}\big((0,1)\setminus \tilde{J}^q  \big),\label{thq'}
\end{equation}
where $\tilde{J}_q$ is the finite set of points in $[0,1]$ which is the limit of bad intervals $\tilde{B}_{\e_\ell}^q$ where $|u_{\e_\ell}|\leq 1 - 2^{-q-1}$ and dips from $1-2^{-q-1}$ to $1-2^{-q}$. Of course, it could turn out that $\tilde{J}^q=\emptyset$, in which case the convergence of $\theta_{\e_\ell}$ to $\theta$ occurs weakly in $H^1_{\rm{loc}}\big((0,1) \big).$ We also note that 
\begin{equation}\label{rhoconv'}
\rho_\e^2 \to 1 \textup{ in }L^2((0,1)),
\end{equation}
which combined with \eqref{thq'} implies that for any $K \subset \subset (0,1) \setminus J^q$
\begin{equation}\label{eq1}
  \lim_{\ell \to \infty}  \int_K \rho_{\e_\ell}^2\theta_{\e_\ell}' \,dx = \int_K  \theta'\,dx.
\end{equation}
Then, using \eqref{eq1}, the weak convergence of $\theta_{\e_\ell}'$ to $\theta'$, and the fact that $\rho_{\e_\ell}\geq 1 - 2^{-q}$ on $K$ for large $\ell$, we can estimate
\begin{align}\notag
    \liminf_{\ell \to \infty} \int_0^1\displaystyle (\mathcal{T}(u_{\e_\ell})-2\pi N)^2\,dx &\geq \liminf_{\ell \to \infty} \int_K\displaystyle \rho_{\e_\ell}^4(\theta_{\e_\ell}')^2-4\pi N\rho_{\e_\ell}^2\theta_{\e_\ell}' +4\pi^2N^2\,dx  \\ \notag
    &\geq \liminf_{\ell \to \infty}\int_K (1-2^{-q})^4(\theta_{\e_\ell}')^2-4\pi N\rho_{\e_\ell}^2\theta_{\e_\ell}' +4\pi^2N^2\,dx \\ 
    &\geq \int_K (1-2^{-q})^4(\theta')^2-4\pi N \theta' + 4\pi^2 N^2 \,dx.
\end{align}
Choosing larger and larger $K$ and using that $\mathcal{H}^0(J^q) < \infty$, we find 
\begin{equation}\notag
\liminf_{\ell \to \infty} \int_0^1\displaystyle (Tw(u_{\e_\ell})-2\pi N)^2\,dx\geq \int_0^1 (1-2^{-q})^4(\theta')^2-4\pi N \theta' + 4\pi^2 N^2 \,dx.
\end{equation}
Finally, sending $q \to \infty$ yields
\begin{equation}\label{twistlsc'1}
\liminf_{\ell \to \infty} \int_0^1\displaystyle (\mathcal{T}(u_{\e_\ell})-2\pi N)^2\,dx\geq \int_0^1 (\mathcal{T}(u)-2\pi N)^2 \,dx.
\end{equation}
Combining \eqref{countlsc} with \eqref{twistlsc'1} completes the proof of lower semi-continuity.\par
Moving on now to the construction of the recovery sequence for any $u\in L^2\big((0,1);\R^2\big),$ if $u\not\in H^1((0,1)\setminus J;S^1)$ for any finite set $J$, then $E_0(u)=\infty$ and taking the trivial recovery sequence $v_\e\equiv u$ will suffice.

Thus we may assume $u\in H^1((0,1)\setminus J;S^1)$ for a finite set $J$ and our task is to construct a sequence $\{v_\e\}\subset H^1\big((0,1);\R^2\big)$ such that
\begin{equation}
v_\e\to u\;\mbox{in}\;L^2\big((0,1);\R^2\big)\quad\mbox{and}\quad \lim_{\e\to 0} E_\e(v_\e)=E_0(u).
\label{recovery}
\end{equation}

In case the traces of $u$ satisfy the desired boundary conditions for admissibility in $E_\e$, that is, in case
$u(0^+)=1$ and $u(1^-)=e^{i\alpha}$ so that $x=0$ and $x=1$ do not lie in $J$, our construction will take the form $v_\e=\rho_\e u$ for a sequence $\{\rho_\e\}\subset H^1\big((0,1);[0,1]\big)$ to be described below. We first describe the construction for this case and then discuss how it is slightly altered in case $0$ or $1$ lie in $J$. Denoting $J$ by $\{x_1,x_2,\ldots,x_k\}$ with then $J\subset (0,1)$ by assumption, we
then  take $\rho_\e$ to satisfy the following conditions:
\begin{enumerate}[(i)]
    \item $\rho_\e$ is smooth on $[0,1]$.
    \item $\rho_\e  \equiv 0$ on $(x_j-\e^2,x_j+\e^2)$.
    \item $\rho_\e$ makes a standard Modica-Mortola style transition from $1$ to $0$ on $I_j^1$, an interval of size say $O(\sqrt{\e})$ with right endpoint $x_j-\e^2$, and makes a transition from $0$ back to $1$ on an interval of size $O(\sqrt{\e})$ with left endpoint  $x_j+\e^2$ that we denote by $I_j^2$, cf. \cite{MM}. 
    \item $\rho_\e\equiv 1$ on $(0,1) \setminus \cup_j ( I_j^1 \cup (x_j-\e^2,x_j+\e^2) \cup I_j^2)$.
\end{enumerate}
In case either $u(0^+)\not=1$ or $u(1^-)\not=e^{i\alpha}$ so that $0$ and/or $1$ lies in $J$, this procedure must be slightly altered near the endpoints. For example, if $u(0^+)\not=1$ then one requires $\rho_\e$ to make a Modica-Mortola style transition from $1$ down to $0$ on the interval $[0,\sqrt{\e}]$, $\rho_\e\equiv 0$ on $[\sqrt{\e},\sqrt{\e}+\e^2]$ and a Modica-Mortola transition from $0$ back up to $1$ on $[ \sqrt{\e}+\e^2,2\sqrt{\e}+\e^2].$ Then we define
\[
\theta_\e(x)=\left\{\begin{matrix} 0&\;\mbox{if}\;x\in [0,\sqrt{\e}+\e^2/2)\\
\theta&\;\mbox{if}\;x>\sqrt{\e}+\e^2/2,\end{matrix}\right.
\]
where $u=e^{i\theta}$, and take $v_\e=\rho_\e e^{i\theta_\e}.$ A similar recipe is taken in a neighborhood of $x=1$ in case $u(1^-)\not=e^{i\alpha}.$

Computing the transition energy of such a construction is classical and can be found in e.g. \cite{ModicaARMA,MM}.  One finds from conditions (ii)-(iv), that 
\begin{equation}\notag
    \int_0^1  \frac{\e}{2}(\rho_\e ')^2 + \frac{1}{4\e}(\rho_\e^2-1)^2 \, dx \to \frac{2\sqrt{2}}{3} \mathcal{H}^0(J).
\end{equation}
Furthermore, since $\theta \in H^1\big((0,1)\setminus J)$, and $\rho_\e\to 1$ in $L^2\big((0,1)\big)$ it is easily seen that
\begin{equation}\notag
   \lim_{\e\to 0} \int_0^1 \frac{\e}{2}\rho_\e^2(\theta_\e')^2\,dx=0\quad\mbox{and}\quad \lim_{\e\to 0} \frac{L}{2}\int_0^1 \mathcal{T}(v_\e)\,dx= \frac{L}{2}\mathcal{T}(u) \, dx.
\end{equation} 
The proof of \eqref{recovery} is complete.
\end{proof}

We observe that for $u\in H^1((0,1)\setminus J;S^1)$, one has
\begin{equation}
E_0(u)=\frac{L}{2}\int_0^1 (\theta'-2\pi N)^2\,dx + \frac{2\sqrt{2}}{3}\mathcal{H}^0(J),\label{thetaE}
\end{equation}
where $u=e^{i\theta}$ for $\theta\in H^1((0,1)\setminus J)$.
Using this formulation, it is then straight-forward to identify the global minimizers of the $\Gamma$-limit, and consequently the limits of global minimizers of $E_\e$ as well:
\begin{theorem}\label{e0mins}
The global minimizer(s) of $E_0$ are given by:
\begin{enumerate}[(i)]
\item the function 
\begin{equation}u(x)=e^{i(2\pi N+\alpha)x}\label{nojumpmin1}
\end{equation}
having constant twist and no jumps when 
\begin{equation}\label{nojumpmincond1}
L\alpha^2<\frac{4\sqrt{2}}{3} \textit{ and } \alpha \in [0,\pi]. 
\end{equation}
\item the function 
\begin{equation}u(x)=e^{i(2\pi (N-1)+\alpha)x}\label{nojumpmin2}
\end{equation}
having constant twist and no jumps when 
\begin{equation}\label{nojumpmincond2}
L(2\pi - \alpha)^2<\frac{4\sqrt{2}}{3} \textit{ and }  \alpha\in [\pi,2\pi).
\end{equation}
\item the one-parameter set of functions given by
\begin{equation}
u(x)=\left\{\begin{matrix}e^{i2\pi Nx}&\;\mbox{if}\;x<x_0,\\
e^{i(2\pi Nx+\alpha)}&\;\mbox{if}\;x>x_0,
\end{matrix}
\right.\label{onejump}
\end{equation}
for any $x_0\in (0,1)$, that have
 one jump and twist $2\pi N$ away from the jump when
\begin{equation}\label{onejumpcond}
L\alpha^2>\frac{4\sqrt{2}}{3} \textit{ and }L(2\pi- \alpha)^2 > \frac{4\sqrt{2}}{3}.
\end{equation}
\end{enumerate}
\end{theorem}
Since any limit of global minimizers of a $\Gamma$-converging sequence must itself be a global minimizer of the $\Gamma$-limit, one immediately concludes the following result based on Theorem \ref{e0mins} and the compactness result Theorem \ref{cpt}:

\begin{corollary}\label{threepos}
Let $\{u_\e\}$ denote a family of minimizers of $E_\e$ subject to the boundary conditions \eqref{bc}. Then if \eqref{onejumpcond} holds, we have $u_\e\to u$ in $L^2$ for some $u$ in the one-parameter family given by \eqref{onejump}, while if \eqref{nojumpmincond1} or \eqref{nojumpmincond2} holds, there will be a subsequence $u_{\e_j}\to u$ in $L^2$ with $u=e^{i(2\pi N+\alpha)x}$ or $u=e^{i(2\pi (N-1)+\alpha)x}$, respectively.
\end{corollary}
\begin{remark}
It is in the case where $L\alpha^2>\frac{4\sqrt{2}}{3} $ and $L(2\pi - \alpha)^2 > \frac{4\sqrt{2}}{3} $ that one really sees the most dramatic effect of the assumption of disparate elastic constants present in our model. The relatively expensive cost of twist leads the global minimizer of $E_\e$, which of course is necessarily smooth, to rapidly change its phase, a process that can only be achieved with small energetic cost by having the modulus simultaneously plunge towards zero.
\end{remark}

\begin{remark}
We have not attempted to determine the optimal location of the jump location $x_0$ for minimizers of $E_\e$ in scenario \eqref{onejump}. We suspect this might entail much higher order energetic considerations--perhaps even at an exponentially small order--but we are not sure. 
\end{remark}

\begin{proof}[Proof of Theorem \ref{e0mins}]
 When $\alpha=0$ then clearly the global minimizer is uniquely given by $u=e^{i2\pi Nx}$ since it has zero energy. Consider then the case $\alpha\in (0,2\pi).$ By selecting any point $x_0\in (0,1)$, and taking 
$u$ to be given by \eqref{onejump},
we see that there is always a competitor with one jump having energy given simply by $\frac{2\sqrt{2}}{3}$. Any competitor jumping more than once has energy no lower than twice that value.  On the other hand, minimization of $E_0$ among competitors with $J=\emptyset$ is standard, since criticality implies $\theta'$ is constant. Given the boundary conditions, this requires $u=e^{i(2\pi M+\alpha)x}$ for some $M \in \mathbb{Z}$ to be determined. The energy of such a $u$ is $\frac{L}{2}(2\pi(M-N)+\alpha)^2$. Since $\alpha \in (0,2\pi)$, the minimum over $M$ is $\frac{L}{2}(2\pi(N-N)+\alpha)^2=\frac{L}{2}\alpha^2$ if $\alpha<2\pi - \alpha$ and $\frac{L}{2}(2\pi(N-1-N)+\alpha)^2=\frac{L}{2}(2\pi - \alpha)^2$ if $2\pi - \alpha < \alpha$. Comparing these two energies to that of the one-jump competitors in \eqref{onejump}, the theorem follows.  We note that if $\alpha=\pi$ in this regime, there are two global minimizers.
\end{proof}

Next we state a result on local minimizers of the $\Gamma$-limit. These functions are the $\e\to 0$ limit of the non-vanishing local minimizers captured in Theorem \ref{constrainedmins}.
\begin{theorem}\label{localmins}
For any positive integer $M$ the function
$u=e^{i(2\pi M+\alpha)x}$ 
is an isolated $L^2$-local minimizer of $E_0$. 
\end{theorem}

By invoking Theorem 4.1 of \cite{KS}, one can conclude from Theorem \ref{localmins} and Theorem \ref{gconv} that there exist local minimizers of $E_\e$ for $\e$ small that converge to this isolated local minimizer of $E_0$. This provides for an alternative proof of existence for these local minimizers to the one given in Proposition \ref{constrainedmins}. However, the approach in Theorem \ref{constrainedmins} yields much more detailed information on the structure of these functions via \eqref{suprho}, \eqref{thetaprime} and \eqref{eq:ls}.

\begin{proof} [Proof of Theorem \ref{localmins}]
We fix a positive integer $M$ and a number $\alpha\in [0,2\pi)$. We will consider the case $M<N$. The case $M\geq N$ is similar. Of course, in case $M=N$ and \eqref{nojumpmincond1}  holds, then in fact $u=e^{i(2\pi N+\alpha)x}$  is the global minimizer, as was already addressed in Theorem \ref{e0mins}. Let us denote $\theta_M:=2\pi Mx+\alpha x$. In light of \eqref{thetaE}, our goal is to show that for some $\delta>0$, one has $E_0(\theta)>E_0(\theta_M)$ whenever $\theta\in H^1\big((0,1)\setminus J;S^1\big)$ for some finite set $J$ provided
$0<\norm{\theta-\theta_M}_{L^2(0,1)}<\delta$.

We begin with the easiest case where $J=\emptyset$ and where $\theta(0)=\theta_M(0),\;\theta(1)=\theta_M(1)$. Writing $v:=\theta-\theta_M$, we calculate
\begin{eqnarray*}
&&E_0(\theta)-E_0(\theta_M)=\frac{L}{2}\int_0^1\big(\theta_M'+v'-2\pi N  \big)^2\,dx-
\frac{L}{2}\int_0^1\big(\theta_M'-2\pi N  \big)^2\,dx\\
&& =2\pi L (M-N+\alpha/2\pi)\int_0^1 v'\,dx+\frac{L}{2}\int_0^1 (v')^2\,dx=\frac{L}{2}\int_0^1 (v')^2\,dx>0,
\end{eqnarray*}
since in the case under consideration, $v(0)=0=v(1)$.

Now we turn to the general case where $J\not=\emptyset$. To this end, consider a competitor $\theta\in H^1\big(\cup_{j=1}^\ell(a_j,b_j)\big)$ where then $J=(0,1)\setminus \cup_{j=1}^\ell(a_j,b_j)$, along with perhaps $x=0$ and/or $x=1$, depending upon whether a competitor satisfies the boundary conditions.  Thus, depending upon the boundary conditions of a competitor, we note that
\begin{equation}
\mathcal{H}^0(J)\in \{\ell-1,\ell,\ell+1\}.\label{ell}
\end{equation}
Again we introduce $v:=\theta-\theta_M$ and after a rearrangement of the indices, we suppose that for  $j=1,2\ldots,\ell'$, one has the condition 
\begin{equation}
    k_j:=v(b_j)-v(a_j)< \frac{\sqrt{2}}{6\pi N}:=k_0,\label{smalljump}
\end{equation}
while for $j=\ell'+1,\ldots,\ell$, the opposite inequality holds. We allow for the possibility that either $\ell'=0$ or $\ell'=\ell$.

Then we again calculate the energy difference $E_0(\theta)-E_0(\theta_M)$ by splitting up the sum as follows:
\begin{eqnarray}
&&E_0(\theta)-E_0(\theta_M)= \frac{2\sqrt{2}}{3}\mathcal{H}^0(J)+2\pi L (M-N+\alpha/2\pi)\sum_{j=1}^{\ell'}\int_{a_j}^{b_j}v'\,dx \nonumber\\
&&+2\pi L (M-N+\alpha/2\pi)\sum_{j=\ell'+1}^{\ell}\int_{a_j}^{b_j}v'\,dx
+\frac{L}{2}\sum_{j=1}^\ell\int_{a_j}^{b_j}(v')^2\,dx\nonumber\\
&& > \frac{2\sqrt{2}}{3}\mathcal{H}^0(J)-2\pi L N k_0 \ell'\nonumber\\
&&-2
\pi LN \sum_{j=\ell'+1}^{\ell}k_j+\frac{L}{2}\sum_{j=\ell'+1}^{\ell}
\int_{a_j}^{b_j}(v')^2\,dx\nonumber\\
&&>\sum_{j=\ell'+1}^{\ell}\big(-2\pi LN k_j+\frac{L}{2}
\int_{a_j}^{b_j}(v')^2\,dx\big),\label{sofar}
\end{eqnarray}
in light of \eqref{smalljump} and \eqref{ell}.

If $\ell'=\ell$ then the last sum is vacuous and the proof is complete. If not, then  we now fix any $j\in \{\ell'+1,\ldots,\ell\}$ for which the reverse inequality to \eqref{smalljump} holds, and observe that
\begin{equation}
   \delta_j^2:= \int_{a_j}^{b_j}v^2\,dx\geq \int_{(a_j,b_j)\cap\{\abs{v}>k_j/4\}}v^2\,dx
    \geq \frac{k_j^2}{16}{\rm meas}\,\big((a_j,b_j)\cap\{\abs{v}>k_j/4\}\big).
    \label{l2big}
\end{equation}
Also,
\begin{eqnarray*}
&&\frac{k_j}{4}<\int_{(a_j,b_j)\cap\{\abs{v}>k_j/4\}}\abs{v'}\,dx\\
&&\leq  {\rm meas}\,\big((a_j,b_j)\cap\{\abs{v}>k_j/4\}\big)^{1/2}\big( \int_{(a_j,b_j)\cap\{\abs{v}>k_j/4\}}(v')^2\,dx\big)^{1/2}.
\end{eqnarray*}
Combining this with \eqref{l2big} yields the inequality
\[
\int_{a_j}^{b_j}(v')^2\,dx\geq \frac{k_j^4}{256\delta_j^2}
\]
which we now substitute into \eqref{sofar} to conclude that
\begin{equation}
E_0(\theta)-E_0(\theta_M)>
\sum_{j=\ell'+1}^{\ell}\big(\frac{Lk_j^4}{512\delta_j^2}-2\pi LNk_j\big).
\label{atlast}
\end{equation}
Choosing $\delta$ (which we recall denotes ($\norm{v}_{L^2(0,1)}$) such that
\[
\delta^2<\frac{k_0^3}{1024\pi  N},
\]
and using that $\delta_j\leq \delta$ while $k_j\geq k_0$ for all $j$, we obtain positivity of the right-hand side of \eqref{atlast}.

\end{proof}

\section{An energy barrier leading to saddle points}

The local minimizers provided by Theorem \ref{constrainedmins} can be viewed as the least energy critical points of $E_\e$ within a given degree or winding number class given by the amount of twist. One might anticipate then that to pass continuously from one of these classes to another requires both the emergence of a zero in the order parameter and the expenditure of a certain amount of energy. What is more, one might expect the presence of saddle points in some sense interspersed between the distinct degree classes. That is the content of the two results in this section.

In the first theorem we demonstrate that the energy barrier between any two local minimizers $u_{\e,M_1}$ and $u_{\e,M_2}$ with $M_1\neq M_2$ is at least $\frac{2\sqrt{2}}{3}$ when $\e$ is sufficiently small. To this end, given a $\Lambda>0,$ we define the energy sublevel set
\[E_{\varepsilon}^\Lambda:=\left\{u\in \mathcal{A}_\alpha:E_\varepsilon(u)<\Lambda \right\}.\]
We have the following: 
\begin{theorem}
\label{thm:sublevel}
Let $M_1,M_2\in\mathbb{N}$ be such that $M_1\neq M_2$ and assume that $u_{\e,M_1}$ and $u_{\e,M_2}$ are local minimizers of $E_\e$ as obtained in Theorem~\ref{constrainedmins}. Suppose that 
\begin{equation}\gamma^\e:[0,1]\to \mathcal{A}_\alpha\quad\mbox{with}\;\gamma^\e(0)=u_{\e,M_1}\;\mbox{and}\;\gamma^\e(1)=u_{\e,M_2}\label{path}\end{equation}
is a continuous path in $\mathcal{A}_\alpha$ that connects $u_{\e,M_1}$ and $u_{\e,M_2}$. Fix an $h>0$ and set $\Lambda_h:=2\pi^2L{(N-M_1-\alpha/2\pi)}^2+\frac{2\sqrt2}3-h$. There exists an $\e_h>0$ such that the curve $\gamma^\e$ leaves the set $E^{\Lambda_h}_\e$ whenever $\e<\e_h.$
\end{theorem}
\begin{proof}
Fix any $h\in (0,1)$ and any curve $\gamma^\e$ satisfying \eqref{path}.
 Denote 
\[\gamma^\e(t):=u^\e_t=\rho^\e_te^{i\theta^\e_t}\]
for every $t\in[0,1]$. The non-vanishing functions $e^{-i\alpha x}u_{\e,M_1}$ and $e^{-i\alpha x}u_{\e,M_2}$ have winding numbers $M_1$ and $M_2$ respectively on $[0,1]$ and so $u^\e_t(x)$ has to vanish for some $x\in(0,1)$ and $t\in(0,1)$. Since $\gamma^\e$ is continuous and $u^\e_t(\cdot)$ is a continuous function for every $t\in[0,1]$, it follows that, given any $\delta\in(0,1/2),$ we can find $t^\e_\delta\in(0,1)$ such that $\min_{x\in(0,1)}\rho^\e_{t^\e_\delta}(x)=\delta$ and the winding number for $e^{-i\alpha x}u^\e_{t^\e_\delta}$ is still equal to  $M_1.$

Now suppose by way of contradiction that $\gamma^\e([0,1])\subset  E_\e^{\Lambda_h}$. We would like to estimate $E_\e(u^\e_{t^\e_\delta}).$ First, by minimizing $E_\e(\rho^\e_{t^\e_\delta}e^{i\theta})$ over $\theta\in \mathcal{T}_{M_1,\alpha}$, note that the same approach that led to \eqref{beth} can be followed to show that there exists a ${\bar\theta}_\e\in{\mathcal T}_{M_1,\alpha}$ such that 
\begin{equation}
{\bar\theta}_\e^\prime=\frac{2\pi LM_1+L\alpha+2\pi LN((\rho^\e_{t^\e_\delta})^2-1)}{L(\rho^\e_{t^\e_\delta})^4+\e {(\rho^\e_{t^\e_\delta})^2}}+O(\sqrt{\e})\label{beth1}
\end{equation}
on $(0,1)$, and necessarily
\begin{equation}
    \label{eq:b3}
    E_\e\left(\rho^\e_{t_\delta}e^{i\bar\theta_\e}\right)\leq E_\e\left(u^\e_{t_\delta}\right).
\end{equation}
Using the standard Modica-Mortola arguments, we now have
\begin{equation*}
    \int_0^1  \frac{\e}{2}((\rho^\e_{t^\e_\delta})')^2 + \frac{1}{4\e}\big((\rho^\e_{t^\e_\delta})^2-1\big)^2 \, dx \geq c(\delta),
\end{equation*}
where $\lim_{\delta\to0}c(\delta)=\frac{2\sqrt{2}}{3}.$ Further, we can appeal to \eqref{beth1}-\eqref{eq:b3} and the assumption that $\gamma^\e(t^\e_{\delta})\in  E_\e^{\Lambda_h}$ to show that
\begin{multline}
    \label{eq:b2}
    \frac{L}{2}\int_0^1\left(2\pi N - (\rho^\e_{t^\e_\delta})^2{\bar\theta}^\prime_\e\right)^2 \, dx=\frac{L}{2}\int_0^1\left(2\pi N - \frac{2\pi LM_1+L\alpha+2\pi LN\big({(\rho^\e_{t^\e_\delta}})^2-1\big)}{L(\rho^\e_{t^\e_\delta})^2+\e}\right)^2 \, dx+O(\sqrt{\e})\\
    =2\pi^2L\int_0^1\left(\frac{N-M_1-\alpha/2\pi+(\e/L) N}{(\rho^\e_{t^\e_\delta})^2+\e/L}\right)^2 \, dx+O(\sqrt{\e})\\=2\pi^2L{(N-M_1-\alpha/2\pi)}^2+O(\sqrt{\e}).
\end{multline}
It then follows from \eqref{eq:b3} that
\begin{equation*}
    E_\e(u^\e_{t^\e_\delta})\geq2\pi^2L{(N-M_1-\alpha/2\pi)}^2+c(\delta)+O(\sqrt{\e}).
\end{equation*}
It is clear, however, that one can select a positive $\delta$ sufficiently small, and then an $\e_h>0$ such that the last expression exceeds $\Lambda_h$ whenever $\e<\e_h$.
\end{proof}

The energy threshold provided by Theorem \ref{thm:sublevel} leads to a straight-forward application of the Mountain Pass Theorem to establish 
saddle points for $E_\e$.

\begin{Theorem}
\label{MPT} For every positive integer $M$ and $\alpha\in [0,2\pi)$ there exists a critical point $v_\eps$ of $E_\e$ within the class $\mathcal{A}_\alpha$. Furthermore, the corresponding critical value $E_\e(v_\e)$ satisfies the asymptotic condition
\begin{equation}
 E_\e(v_\e)\to 2\pi^2L{(N-M-\alpha/2\pi)}^2+\frac{2\sqrt2}3\quad\mbox{as}\;\e\to 0.   \label{asympcrit}
\end{equation}
\end{Theorem}
\begin{proof}
 First, we note that the arguments in Theorem \ref{thm:sublevel} can easily be adapted with the same energy threshold to a curve that connects the states $U_{M}:=e^{i(2\pi M+\alpha)x}$ and $U_{M_1}:=e^{i(2\pi M_1+\alpha)x}$ for any two positive integers $M$ and $M_1$. Fixing $\e>0$ one defines the potential critical value $c_\e$ via
\[
c_\e:=\inf_{\gamma\in \Gamma_\e}\;\max_{t\in [0,1]}\;E_\e(\gamma(t)),
\]
where $\Gamma_\e$ is the set of continuous curves $\gamma$ such that
\begin{equation}\gamma:[0,1]\to \mathcal{A}_\alpha\quad\mbox{with}\;\gamma(0)=U_{M}\;\mbox{and}\;\gamma(1)=U_{M+1}.\label{path1}\end{equation}
Beginning with the case $M<N$ we have that
\[
E_\e(U_{M})\sim 
2\pi^2L{(N-M-\alpha/2\pi)}^2
\]
while
\[
E_\e(U_{M+1})\sim 
2\pi^2L{(N-M-1-\alpha/2\pi)}^2,
\]
so that, in particular, $E_\e(U_{M+1})<E_\e(U_{M})$.
Then the implication of Theorem \ref{thm:sublevel} is that $E_\e$ exhibits the requisite mountain pass structure since for any $h>0$ one has 
\begin{equation}
\max_{t\in [0,1]}\;E_\e(\gamma(t))\geq 
2\pi^2L{(N-M-\alpha/2\pi)}^2+\frac{2\sqrt2}3-h>E_\e(U_{M})>E_\e(U_{M+1})\label{mps}
\end{equation}
for any $\gamma\in \Gamma_\e$, provided $\e$ is sufficiently small.

Subtracting off the boundary conditions by writing any competitor $u\in \mathcal{A}_{\alpha}$ as $u=\tilde{u}+\ell(x)$ where
$
\ell(x):=1+x\big(e^{i\alpha}-1\big),
$
we can work in the space $H^1_0\big((0,1)\big)$. It remains to verify the Palais-Smale condition. Under assumptions
\[
E_\e(\tilde{u}_k+\ell)<C_0\quad\mbox{and}\quad
\norm{\delta E_\e(\tilde{u}_k+\ell)}\to 0\quad\mbox{as}\;k\to 0,
\]
for $\{\tilde{u}_k\}\subset H^1_0\big((0,1)\big)$, it immediately follows from the uniform energy bound that after passing to a subsequence (with notation suppressed), one has 
\begin{equation}
\tilde{u}_k\rightharpoonup \tilde{u}_{\e,M}\;\mbox{weakly in}
\;H^1\quad\mbox{and}\quad
\tilde{u}_k\to \tilde{u}_{\e,M}\;\mbox{uniformly}\quad\mbox{as}\;k\to \infty,\label{ps}
\end{equation}
for some $\tilde{u}_{\e,M}\in H^1_0$.
Then one writes $E_\e$ as the sum of the Allen-Cahn energy and the twist energy $I(\tilde{u}):=\frac{L}{2}\int_0^1 \mathcal{T}(\tilde{u}+\ell)\,dx$ and one follows the standard proof used to verify that the Allen-Cahn functional satisfies Palais-Smale (see e.g. \cite{JS}, Prop. 3.3). The key step in upgrading the weak $H^1$ convergence to strong convergence is writing out the difference
$\delta E_\e(\tilde{u}_k+\ell;\tilde{u}_k)-\delta E_\e(\tilde{u}_k+\ell;\tilde{u}_{\e,M})$, and in light of the convergences \eqref{ps}, the extra twist terms in this difference pose no additional trouble. We conclude from the Mountain Pass Theorem that a critical point $v_{\e,M}:=\tilde{u}_{\e,M}+\ell$ exists with $E_\e(v_{\e,M})=c_\e$. 

Now we turn to the proof of condition \eqref{asympcrit}. Again, we know from Theorem \ref{thm:sublevel} that for any $h>0$, one has the inequality \eqref{mps} for $\e$ small enough, so that 
\begin{equation}
\liminf c_\e\geq 
2\pi^2L{(N-M-\alpha/2\pi)}^2+\frac{2\sqrt2}3\label{liminf}
\end{equation}
On the other hand, we can build a continuous path $\gamma^\e:[0,1]\to\mathcal{A}_\alpha$ as follows:\\
1) Writing $U_{M}=e^{i\theta_{M}}$ as $t$ varies between $0$ and say $1/3$, the modulus gradually depresses towards $0$ in a small interval of $x$-values about $x=1/2$ via the standard Modica-Mortola construction, so that $\gamma^\e(1/3)\equiv 0$ for say $1/2-\e^2\leq x\leq 1/2+\e^2$. For this interval of $t$-values one leaves the phase $\theta_{M}$ unchanged. As we have previously noted, such a procedure can be executed with
\[
E_\e(\gamma^\e(t))\leq 2\pi^2L{(N-M-\alpha/2\pi)}^2+\frac{2\sqrt2}3+O(\e)\;\mbox{for}\;t\in [0,1/3).
\]
Of course along the subinterval where the modulus vanishes, the value of the phase is irrelevant but we find it convenient in this exposition to define the phase throughout the whole interval for each function $\gamma^\e(t)$.\\
2) At $t=1/3$ we introduce a removable discontinuity in the phase $\theta_{M}$ at $x=1/2$ where the modulus vanishes. Then, as $t$ increases from $t=1/3$ to $t=2/3$, one takes the phase to gradually converge to $\theta_{M+1}$ and $\theta_{M+1}-2\pi$ on $[0,1/2)$ and $(1/2,1]$, respectively, while leaving the modulus unchanged. Since $M<N$, the $O(1)$ energy contribution of the twist will decrease under this process. As $t$ approaches $2/3$, we converge to $U_{M+1}$ except for the small  interval about $x=1/2$ where the modulus is depressed.\\
3) In the time interval $t\in [2/3,1]$ one smoothly raises the modulus back up to $1$ on $1/2-\e^2\leq x\leq 1/2+\e^2$ so that at $t=1$ one has $\gamma^\e(1)=U_{M+1}$.
Again, this process decreases energy so that throughout the interval $0\leq t\leq 1$ one maintains the estimate
\[
E_\e(\gamma^\e(t))\leq 2\pi^2L{(N-M-\alpha/2\pi)}^2+\frac{2\sqrt2}3+O(\e).
\]
Hence, we conclude that
\[
\limsup{c_\e}\leq \limsup{E_\e(\gamma^\e)}\leq
2\pi^2L{(N-M-\alpha/2\pi)}^2+\frac{2\sqrt2}3
\]
and together with \eqref{liminf} we arrive at \eqref{asympcrit}.
\end{proof}
\section{The case of unbounded twist}

Finally, we consider the situation of an energy that encourages more and more twist in the $\e\to 0$ limit. To this end, we replace $N$ in \eqref{Eeps} by $N_\eps:=1/\e^\beta$ where $\beta$ is a positive number chosen less than $1/2$ in order to retain an energy bound that is uniform in $\e$. Thus, we study global and local minimizers of an energy $\tilde{E}_\e$ given by
\begin{equation}
\tilde{E}_\e(u)=\int_0^1 \frac{\e}{2}\abs{u'}^2+\frac{1}{4\e}(\abs{u}^2-1)^2+
\frac{L}{2}(u_1\,u_2'-u_2\,u_1'-2\pi \e^{-\beta})^2\,dx,\label{tEeps}
\end{equation}
again subject to the boundary conditions
$u(0)=1,\; u(1)=e^{i\alpha}$ for some $\alpha\in [0,2\pi)$.

Of course existence of global minimizers for each $\e>0$ follows as in Theorem \ref{GE}. One also can establish a version of the local minimizer result Theorem \ref{constrainedmins}:

\begin{Theorem}\label{locals}
Fix any positive integer $m$ and any $\alpha\in [0,2\pi)$. Then there exists an $\e_0>0$ such that for all $\e<\e_0$ there exist non-vanishing local minimizers $u_{\e,\pm}=\rho_{\e,\pm} e^{i\theta_{\e,\pm}}$ of $\tilde{E}_\e$ within the class $\mathcal{A}_\alpha$ such that
\begin{eqnarray}
&&
\limsup\frac{\norm{{\rho_{\e,\pm}}-1}_{L^{\infty}(0,1)}}{\e}<\infty\;\mbox{as}\;\e\to 0
\label{suprho1}\\
&&\mbox{and}\\
&&\theta_{\e,\pm}'\to 2\pi\left(\floor*{\e^{-\beta}}\pm m\right) +\alpha\;\mbox{as}\;\e\to 0\;\mbox{uniformly in}\;x\in [0,1].\label{thetaprime1}
\end{eqnarray}
\end{Theorem}
\begin{proof}
The proof follows along similar lines as the proof of Theorem \ref{constrainedmins}. First define $M^{\pm}_\e=\floor*{\e^{-\beta}}\pm m$. Then one writes competitors for constrained minimization of $\tilde{E}_\e=\tilde{E}_\e(\rho,\theta)$ in polar form $(\rho,\theta)$ where $\rho$ satisfies \eqref{mtH} and $\theta(0)=0,\;\theta(1)=2\pi M^{\pm}_\e+\alpha$. The requirement $\beta<1/2$ assures that a version of the uniform energy bound \eqref{test}
 still holds. Similarly, a uniform bound on the constant of integration $C_\e$ is achievable as in \eqref{Ceform}, with the bound now depending on $m$. The rest of the argument is unchanged.\end{proof}
 Next we consider the asymptotic behavior as $\e \to 0$ of $\tilde{E}_\e$. Due to the fact that $\e^{-\beta} \to \infty$ as $\e \to 0$, we expect that the elements of an energy bounded sequence will oscillate more and more rapidly as $\e \to 0$. 
\begin{theorem}\label{macro}
Suppose that for some $0 < \beta < 1/2$, $\{u_\e \}\subset \mathcal{A}_\alpha$ satisfies the uniform energy bound
\begin{equation}\label{unibound}
\tilde{E}_\e(u_\e) \leq C_0 < \infty.
\end{equation}
Then $\abs{u_\e}^2\to 1$ in $L^2(0,1)$ and there exists a finite set $J' \subset (0,1)$ and a subsequence $\{u_{\e_{\ell}}\}$ such that for every compact set $K\subset \subset (0,1)\setminus J'$, there exists an $\e_0(K)>0$ such that for every $\e_{\ell} <\e_0$, one has $|u_{\e_{\ell}}|>0$ on $K$ and there is a lifting whereby $u_{\e_{\ell}}=\rho_{\e_{\ell}} e^{2\pi i v_{\e_{\ell}}/\e_{\ell}^\beta}$, with 
\begin{equation}\label{strongv}
    v_{\e_{\ell}} \to x \textit{ strongly in }H^1_{loc}((0,1) \setminus J').
\end{equation}
In addition, we have
\begin{equation}\label{weak}
u_{\e}\rightharpoonup 0\textit{ weakly in }L^2((0,1);\mathbb{C}),
\end{equation}
so that the entire sequence converges weakly to $0$.
\end{theorem}
\begin{proof}[Proof of \Cref{macro}]
By the same argument as the one leading up to \eqref{notalot}, we can identify finite unions of open intervals $B_\e$ such that on $(0,1) \setminus B_\e$, $\rho_\e \geq 2^{-q}$. Also, by restricting to a subsequence $\{\e_{\ell}\}$, we can assume that the sets $B_{\e_{\ell}}$ collapse to a finite set of points $J'$. We may therefore define liftings $\theta_{\e_{\ell}}$ such that on each of the finitely many intervals comprising $(0,1)\setminus B_{\e_{\ell}}$, the value of $\theta_{\e_{\ell}}$ at the left endpoint of an interval is greater than the value of $\theta_{\e_{\ell}}$ at the right endpoint of the previous interval, with a difference of no more than $2 \pi$. Also, we can without loss of generality suppose that $0$ is in the domain of $\theta_{\e_\ell}$ and set $\theta_{\e_\ell}(0)=0$. If we define
\begin{equation}\label{vdef}
    v_{\e_{\ell}}:= \frac{\e_{\ell}^\beta \theta_{\e_{\ell}}}{2\pi},
\end{equation}
then we may rewrite the twist term in terms of $v_{\e_{\ell}}$ and use the uniform energy bound to conclude that
\begin{equation}\label{twbound}
\frac{1}{2}\int_{(0,1)\setminus B_{\e_{\ell}}}  \frac{L}{\e_{\ell}^{2\beta}} \left( \rho_{\e_{\ell}}^2 v_{\e_{\ell}}'-1\right)^2 \leq C_0.  
\end{equation}
Furthermore, due to the choice of $\theta_{\e_{\ell}}$ on each subinterval of $B_{\e_{\ell}}$, we see that
\begin{equation}\label{vjump}
    \textit{the value of $v_{\e_{\ell}}$ jumps by no more than $\e_{\ell}^\beta$}
\end{equation}
from the right endpoint of one subinterval to the left endpoint of the subsequent one. After passing to a further subsequence (with notation suppressed), we conclude from \eqref{twbound} that for any $K_1 \subset \subset [0,1]\setminus J'$,
\begin{equation}\label{vprime}
    v_{\e_{\ell}}' \to 1 \textup{ in }L^2(K_1).
\end{equation}
From \eqref{vjump}, \eqref{vprime}, and the condition $v_{\e_{\ell}}(0)=0$, we deduce that
\begin{equation}\label{ltwov}
    v_{\e_{\ell}} \to x \textup{ in $L^\infty(K_1)$}.
\end{equation}
As in \eqref{nested}, we may repeat this procedure on a nested sequence of compact sets to arrive at a subsequence (still denoted by $v_{\e_{\ell}}$) such that
\begin{equation}\notag
    v_{\e_{\ell}} \to x \textup{ in }H^1_{loc}((0,1) \setminus J').
\end{equation}

To prove \eqref{weak}, we must demonstrate that for any $w\in L^2((0,1);\mathbb{C})$,
\begin{equation}\label{needeq}
    \int_0^1 u_{\e} \overline{w} \,dx \to 0,
\end{equation}
where the bar denotes complex conjugation. Let us first obtain a subsequence $u_{\e_\ell}$ satisfying \eqref{vprime} and \eqref{ltwov} such that $\lim_{\ell \to \infty}\left|\int u_{\e_\ell} \overline{w} \right|$ achieves the limit superior. By Egorov's theorem, after restricting to a subsequence such that $v_{\e_\ell}' \to 1$ almost everywhere, we can assume without loss of generality that $v_{\e_\ell}'\to 1$ almost uniformly on $(0,1)\setminus J'$. Also, if we approximate $w$ by smooth functions, it is enough to show that for any $K \subset \subset (0,1)\setminus J'$ on which $v_{\e_\ell}' \to 1$ uniformly,
\begin{equation}\label{smooth}
    \int_0^1 u_{\e_\ell} \overline{\varphi} \,dx = \int_0^1 \rho_{\e_\ell}e^{2\pi i v_{\e_\ell}/\e^{\beta}_\ell} \overline{\varphi} \,dx\to 0
\end{equation}
if $\varphi \in C_c^\infty(K;\mathbb{C})$. Since $\rho_{\e_\ell} \to 1$ in $L^2$, \eqref{smooth} would follow from the condition
\begin{equation}\notag
    \int_0^1 e^{2\pi i v_{\e_\ell}/\e^{\beta}_\ell}\overline{\varphi} \,dx \to 0,
\end{equation}
which we now prove. Because $v_{\e_\ell}'\to 1$ uniformly on $\mathrm{spt}\, \varphi$, $v_{\e_\ell}$ is a diffeomorphism of $\mathrm{spt}\, \varphi$ for large $\ell$. Thus for $\ell$ large enough, we can add and subtract $\int e^{2\pi i x/\e^{\beta}_\ell}\overline{\varphi}$ and then change variables to get
\begin{align}\notag
    \left|\int_{\mathrm{spt}\,\varphi} e^{2\pi i v_{\e_\ell}/\e^{\beta}_\ell}\overline{\varphi} \right| &\leq \left|\int_{\mathrm{spt}\,\varphi} e^{2\pi i x/\e^{\beta}_\ell}\overline{\varphi}  \right| + \left|\int_{\mathrm{spt}\,\varphi} e^{2\pi i x/\e^{\beta}_\ell}\overline{\varphi} - \int_{\mathrm{spt}\,\varphi} e^{2\pi i v_{\e_\ell}/\e^{\beta}_\ell}\overline{\varphi} \right| \\ \notag
    &=\left|\int_{\mathrm{spt}\,\varphi} e^{2\pi i x/\e^{\beta}_\ell}\overline{\varphi} \right| + \left|\int_{\mathrm{spt}\,\varphi} e^{2\pi i x/\e^{\beta}_\ell} \overline{\varphi} - \int_{v_{\e_\ell}(\mathrm{spt}\,\varphi)} e^{2\pi i x/\e^{\beta}_\ell}\frac{\overline{\varphi}(v_{\e_\ell}^{-1})}{ |v_{\e_\ell}'(v_{\e_\ell}^{-1}) |}\right| .
\end{align}
The first term goes to zero as $\ell \to \infty$ since $e^{2\pi i x/\e^{\beta}_\ell} \rightharpoonup 0$. In the second term we can add and subtract $\int_{v_{\e_\ell}(\mathrm{spt}\,\varphi)} e^{2\pi i x/\e^{\beta}_\ell}\overline{\varphi}(v_{\e_\ell}^{-1})$ and then change variables back, yielding
\begin{align}\notag
     \left|\int_{\mathrm{spt}\,\varphi} e^{2\pi i x/\e^{\beta}_\ell} \overline{\varphi} - \right. & \left. \int_{v_{\e_\ell}(\mathrm{spt}\,\varphi)} e^{2\pi i x/\e^{\beta}_\ell}\frac{\overline{\varphi}(v_{\e_\ell}^{-1})}{ |v_{\e_\ell}'(v_{\e_\ell}^{-1}) |}\right| \\ \notag
    &\leq \int_{v_{\e_\ell}(\mathrm{spt}\,\varphi)\cup\mathrm{spt}\,\varphi} \left|\varphi - \varphi(v_{\e_\ell}^{-1})  \right| + \int_{v_{\e_\ell}(\mathrm{spt}\,\varphi)} \frac{|\varphi(v_{\e_\ell}^{-1})|}{|v_{\e_\ell}'(v_{\e_\ell}^{-1}) |}\left| |v_{\e_\ell}'(v_{\e_\ell}^{-1}) |-1 \right| \\ \notag
    &\leq \|\varphi' \|_{L^\infty}\int_{v_{\e_\ell}(\mathrm{spt}\,\varphi)\cup \mathrm{spt}\,\varphi} |x-v_{\e_\ell}^{-1} | + \|\varphi \|_{L^\infty} \int_{\mathrm{spt}\,\varphi}|v_{\e_\ell}'-1 |,
\end{align}
which approaches zero by \eqref{vprime} and \eqref{ltwov}. 
\end{proof}

We would also like to describe the asymptotic behavior of minimizers in this regime by identifying a limiting problem. As demonstrated in the previous theorem, no meaningful limit can be extracted from simply looking at the sequence $\{u_\e\}$. Instead, we examine the ``microscale" behavior of $u_\e$ by eliminating the excess twist in the limit $\e \to 0$, in the sense that we obtain a limiting asymptotic problem for the rescaled functions
\begin{equation}\notag
    w(x):=u(x) e^{-2\pi i \floor{\e^{-\beta}}x}.
\end{equation}
Here $\floor{\e^{-\beta}}$ denotes the integer part of $\e^{-\beta}$.

In terms of $w$, the energy $\tilde{E}_\e(u)$ is given by
\begin{align}\notag
    \tilde{E}_\e(u) = F_\e(w) &:= \int_0^1 \frac{\e}{2}\abs{(we^{2\pi i \floor{\e^{-\beta}}x})'}^2+\frac{1}{4\e}(\abs{w}^2-1)^2\\ \notag
&\qquad+\frac{L}{2}(w_1\,w_2'-w_2\,w_1'+|w|^22\pi\floor{\e^{-\beta}}-2\pi \e^{-\beta})^2\,dx.\notag
\end{align}
The boundary conditions imposed on competitors for $F_\e$ are the same as those for $\tilde{E}_\e$. The asymptotic behavior of minimizers of $\tilde{E}_\e$ can therefore be completely understood in terms of $F_\e$, so we pursue an asymptotic limit for $F_\e$. Let us define the limiting functional as in \Cref{sec:gamma}, with slightly altered notation to emphasize the dependence on preferred twist:
\begin{align}
\notag
E_{0,A}(w) := \left\{ 
\begin{array}{cc}
\displaystyle     \frac{L}{2}\int_0^1\displaystyle (w_1\,w_2'-w_2\,w_1'-2\pi A)^2\,dx + \frac{2\sqrt{2}}{3}\mathcal{H}^0(J)     &\mbox{if}\; w\in H^1((0,1)\setminus J;S^1) \\ \\
+\infty & \mbox{ otherwise}.
\end{array}
\right.
\end{align}
We recall that $0$ and/or $1$ belongs to $J$ depending on whether or not the traces of $u$ satisfy the desired boundary conditions inherited from $E_\e$; that is, we include $x=0$ in $J$ only if $u(0^+)\not=1$ and we include $x=1$ in $J$ only if $u(1^-)\not=e^{i\alpha}$.
\begin{theorem}\label{gammablowup}
Let $0 < \beta < 1/2$ and suppose that for a subsequence $\{\e_\ell\}\to 0$ and some $A\in [0,1]$ we have
\begin{equation}\notag
    \e_\ell^{-\beta} - \floor{\e_\ell^{-\beta}} \to A.
\end{equation}
Then $\{F_{\e_\ell}\}$ $\Gamma$-converges to $E_{0,A}$ in $L^2\left((0,1);\mathbb{R}^2\right)$.
\end{theorem}
We also have the compactness result
\begin{theorem}\label{cpt2}
If $\{u_\e \}_{\e>0}$ satisfies 
\begin{equation}\label{energybound2}
\tilde{E}_\e(u_\e)=F_\e(w_\e)\leq C_0 <\infty,
\end{equation}
and
\begin{equation}\label{Aconv}
    \e_\ell^{-\beta} - \floor{\e_\ell^{-\beta}} \to A
\end{equation}
for some $0 < \beta < 1/2$, then there exists a function $w\in H^1((0,1)\setminus J';S^1)$ where $J'$ is a finite, perhaps empty, set of points in $(0,1)$ such that along a subsequence $\e_{\ell}\to 0$ one has
\begin{equation}\label{modcon2}
u_{\e_\ell}e^{-2\pi i \floor{\e_\ell^{-\beta}}x}=w_{\e_\ell} \to w\;\mbox{in}\;L^2\big((0,1);\C\big).
\end{equation}
Furthermore, writing $w(x)=e^{i\theta(x)}$ for $\theta\in H^1((0,1)\setminus J')$, we have that for every compact set $K\subset\subset (0,1)\setminus J'$, there exists an $\e_0(K)>0$ such that for every $\e_\ell<\e_0$ one has $\abs{u_{\e_\ell}}=\abs{w_{\e_\ell}}>0$ on $K$ and there is a lifting whereby $u_{\e_\ell}(x)e^{-2\pi i \floor{\e_\ell^{-\beta}}x}=w_{\e_\ell}(x)=\rho_{\e_\ell}(x)e^{i\theta_{\e_\ell}(x)}$ on $K$, with
\begin{equation}
\theta_{\e_{\ell}}\rightharpoonup \theta\;\mbox{weakly in}\;H^1_{loc}\big((0,1)\setminus J'\big).\label{weakcon2}
\end{equation}
\end{theorem}
\begin{proof}[Proof of \Cref{cpt2}]
The proof is based on the proof of \Cref{cpt}. First, we estimate that
\begin{align}\notag
    \int_0^1 \frac{\e}{2}\abs{(w_\e e^{2\pi i \floor{\e^{-\beta}}x})'}^2 \,dx &= \int_0^1 \frac{\e}{2}\left|w_\e'+2\pi i w\floor{\e^{-\beta}}\right|^2\,dx \\
    &= \int_0^1 \frac{\e}{2}\left|w_\e' \right|^2\,dx + O(\e^{1/2-\beta})\,dx
\end{align}
for an energy bounded sequence $\{w_\e \}$. Therefore, 
\begin{align}\notag
F_\e(w_\e) = \int_0^1 \frac{\e}{2}\abs{w_\e '}^2&+\frac{1}{4\e}(\abs{w_\e}^2-1)^2 \\ \label{replace}&+
\frac{L}{2}(\mathcal{T}(w_\e)+|w_\e|^22\pi\floor{\e^{-\beta}}-2\pi \e^{-\beta})^2\,dx + O(\e^{1/2-\beta}).
\end{align}
The rest of the proof follows almost exactly as in \Cref{cpt}. Indeed, the only difference between $E_\e$ in that theorem and the right hand side of \eqref{replace} here is the preferred twist $2\pi N$ versus $|w_\e|^22\pi\floor{\e^{-\beta}}-2\pi \e^{-\beta}$, respectively. For the purpose of showing compactness, this distinction is immaterial, since it is only the uniform boundedness of the preferred twist $2\pi N$ in $L^2$ that was used in \eqref{h1} to obtain compactness. Using $\beta < 1/2$, we can estimate
\begin{align*}
   \left\| |w_\e|^22\pi\floor{\e^{-\beta}}-2\pi\e^{-\beta} \right\|_{L^2} &\leq \left\|(|w_\e|^2-1)2\pi\floor{\e^{-\beta}}\right\|_{L^2}+\left\|2\pi\floor{\e^{-\beta}} - 2\pi\e^{-\beta} \right\|_{L^2}, \\ 
   &\leq 2\pi  \left( \left\|(|w_\e|^2-1)\e^{-1/2}\right\|_{L^2}+1\right) \\ 
   & \leq 2\pi\left( 2\sqrt{C_0} + 1\right),
\end{align*}
so we are done.
\end{proof}
\begin{proof}[Proof of \Cref{gammablowup}]
We begin with the lower-semicontinuity condition. Let $w_\e \to w$ in $L^2$. We can assume that 
\begin{equation}\label{finitelim}
    \liminf_{\e \to 0} F_\e(w_\e) \leq C_0< \infty, 
\end{equation}
otherwise the lower-semicontinuity is trivial. The proof is similar to the proof of \eqref{lscgc} in \Cref{gconv}. Also, due to \eqref{replace}, it is enough to show that 
\begin{align}\notag
\liminf_{\e \to 0}&\int_0^1 \frac{\e}{2}\abs{w_\e'}^2+\frac{1}{4\e}(\abs{w_\e}^2-1)^2 +\frac{L}{2}(\mathcal{T}(w_\e)+|w_\e|^22\pi\floor{\e^{-\beta}}-2\pi \e^{-\beta})^2\,dx\\ &\geq E_{0,A}(w).\label{enough}
\end{align}
\par
First, for the twist term, it must be verified that under the assumption \eqref{finitelim},
\begin{align}\notag
    \liminf_{\e \to 0}\int_0^1 
\frac{L}{2}(\mathcal{T}(w_\e)+|w_\e|^22\pi&\floor{\e^{-\beta}}-2\pi \e^{-\beta})^2\,dx \\ \label{equallim}&\geq \frac{L}{2}\int_0^1\displaystyle (\mathcal{T}(w)-2\pi A)^2\,dx.
\end{align}
In \Cref{gconv}, after \eqref{thq'}, we proved the inequality
\begin{equation}\notag
    \int_K (1-2^{-q})^4(\theta_{\e_\ell}')^2 - 4\pi N(\rho_{\e_\ell})^2\theta_{\e_\ell}' + 4\pi^2N^2 \,dx \geq \int_K (1-2^{-q})^4(\theta')^2 - 4\pi N\theta' + 4\pi^2N^2 \,dx,
\end{equation}
where $K$ is a compact set on which $\theta_{\e_\ell}' \rightharpoonup \theta'$ and $\rho_{\e_\ell} \geq 1-2^{-q}$, followed by an exhaustion argument in $K$ and $q$ to prove lower-semicontinuity of the twist in \eqref{twistlsc'1}. The corresponding inequality to be verified in this case is 
\begin{align}\notag
    \int_K (1-2^{-q})^4(\theta_{\e_\ell}')^2 &+ 4\pi \left( |w_{e_\ell}|^2\floor{\e_\ell^{-\beta}}-\e_\ell^{-\beta}\right)|w_{\e_\ell}|^2\theta_{\e_\ell}'+4\pi^2 \left( |w_{e_\ell}|^2\floor{\e_\ell^{-\beta}}-\e_\ell^{-\beta}\right)^2 \,dx \\ \label{bigone} &\geq \int_K (1-2^{-q})^4(\theta')^2 - 4\pi A\theta' + 4\pi^2 A^2 \,dx,
\end{align}
which is the left-hand side of \eqref{equallim} expanded out and estimated using $|w_{\e_\ell}|\geq 1- 2^{-q}$ on $K$, on which $\theta_{\e_\ell}' \rightharpoonup \theta'$. The desired inequality \eqref{bigone} would follow immediately from the weak convergence of $\theta_{\e_\ell}'$ and the two conditions
\begin{equation}\label{small1}
\e_\ell^{-\beta} - |w_{\e_\ell}|^2\floor{\e_\ell^{-\beta}} \to A \textup{ in }L^2   
\end{equation}
and
\begin{equation}\label{small2}
    |w_{\e_\ell}|^2(\e_\ell^{-\beta} - |w_{\e_\ell}|^2\floor{\e_\ell^{-\beta}}) \to A \textup{ in }L^2,
\end{equation}
which we check in turn. First for \eqref{small1}, we estimate
\begin{align}\notag
    \left\|\e_\ell^{-\beta} - |w_{\e_\ell}|^2\floor{\e_\ell^{-\beta}}-A \right\|_{L^2} \leq \left\|\e_\ell^{-\beta} - \floor{\e_\ell^{-\beta}}-A \right\|_{L^2}+ \left\|(1-|w_{\e_\ell}|^2) \floor{\e_\ell^{-\beta}}\right\|_{L^2} .
\end{align}
The first term goes to zero as $\e \to 0$ due to \eqref{Aconv}, and the second vanishes due to the uniform energy bound \eqref{finitelim}, since $\beta < 1/2$. Moving on to \eqref{small2}, we can repeat the argument \eqref{linfbound} to find that $$\|w_{\e_\ell} \|_{L^\infty} \leq M(C_0).$$ The second condition \eqref{small2} can be shown as consequence of this $L^\infty$ bound, \eqref{small1}, and \eqref{finitelim} after writing
\begin{equation}\notag
    |w_{\e_\ell}|^2(\e_\ell^{-\beta} - |w_{\e_\ell}|^2\floor{\e_\ell^{-\beta}}) - A = |w_{\e_\ell}|^2(\e_\ell^{-\beta} - |w_{\e_\ell}|^2\floor{\e_\ell^{-\beta}} - A) + (|w_{\e_\ell}|^2-1)A.
\end{equation}
Choosing larger and larger $K$ which exhaust $(0,1)$ and letting $q \to \infty$ as in \Cref{gconv}, the proof of \eqref{equallim} is finished. The remainder of the lower-semicontinuity proof follows from the proof of \Cref{gconv} and \eqref{replace}. The recovery sequence is very similar to the proof of \Cref{gconv}, which is evident due to the similarity of \eqref{replace} with $E_\e$, so we omit the details. We only mention that on the set of size $O(\e)$ where $|w_\e|\neq 1$, the assumption $\beta < 1/2$ is needed to make sure the twist term vanishes in the limit $\e \to 0$.
\end{proof}
Finally, we identify the minimizers of $E_{0,A}$. As in \Cref{threepos}, this provides a description of all subsequential limits of a family of minimizers $\{u_\e \}$ for $F_\e$ and thus $\tilde{E}_\e$. We omit the proof since it follows the same strategy as the proof of \Cref{threepos}.
\begin{theorem}\label{blowupmins}
Let $N=N(A,\alpha)$ be the closest integer to $A-\frac{\alpha}{2\pi}$, so that $N \in \{-1,0,1 \}$. Then the global minimizer(s) of $E_{0,A}$ are given by
\begin{enumerate}[(i)]
\item the function 
\begin{equation}u(x)=e^{i(2\pi N+\alpha)x}\label{blowupnojumpmin}
\end{equation}
having constant twist and no jumps when 
\begin{equation}\label{blowupnojumpmincond}
L(2\pi(N-A)+\alpha)^2<\frac{4\sqrt{2}}{3} .
\end{equation}
\item the one-parameter set of functions given by
\begin{equation}
u(x)=\left\{\begin{matrix}e^{i2\pi Ax}&\;\mbox{if}\;x<x_0,\\
e^{i(2\pi Ax+\alpha-2\pi A)}&\;\mbox{if}\;x>x_0,
\end{matrix}
\right.\label{blowuponejump}
\end{equation}
for any $x_0\in (0,1)$, that have
 one jump and twist $2\pi A$ away from the jump, when
\begin{equation}\label{blowuponejumpcond}
L(2\pi(N-A)+\alpha)^2>\frac{4\sqrt{2}}{3} .
\end{equation}
\end{enumerate}
\end{theorem}

\section{Acknowledgments.}
DG acknowledges the support from NSF DMS-1729538. PS acknowledge the support from a Simons Collaboration grant 585520.
\bibliographystyle{acm}
\bibliography{GNS2}
\end{document}